\documentclass[11pt, reqno]{amsart}

\usepackage{graphicx,color,amsmath,amssymb,mathrsfs}
\usepackage{amsfonts,amsthm,epsfig,epstopdf, esint,soul,xfrac}\usepackage[mathscr]{euscript}
\usepackage[margin=0.8in]{geometry}

\newtheorem{theorem}{Theorem}[section]
\newtheorem{lemma}[theorem]{Lemma}
\newtheorem{proposition}[theorem]{Proposition}
\newtheorem{corollary}[theorem]{Corollary}
\newtheorem{remark}[theorem]{Remark}

\usepackage[utf8]{inputenc}
\usepackage[english]{babel}

\usepackage{xcolor}
\usepackage[colorlinks = true,
            linkcolor = blue,
            urlcolor  = blue,
            citecolor = blue,
            anchorcolor = blue]{hyperref}
\newcommand*{\medcup}{\mathbin{\scalebox{1.3}{\ensuremath{\cup}}}}%
\newcommand*{\mlcup}{\mathbin{\scalebox{1}{\ensuremath{\bigcup}}}}%

\newcommand{\F}{\mathcal{F}}

\newcommand{\E}{\mathcal{E}}

\newcommand{\R}{{\mathbb{R}}}
\newcommand{\N}{{\mathbb{N}}}
\newcommand{\e}{{\epsilon}}

\newcommand{\pa}{{\partial}}
\newcommand{\C}{{\mathscr{C}}}

\newcommand{\M}{{\mathcal{M}}}

\newcommand{\Hi}{{\mathcal{H}}}

\newcommand{\W}{\Omega}
\newcommand{\om}{\omega}

\newcommand{\Emod}{{\tilde{E}}}
\newcommand{\JO}{{J_0}}
\newcommand{\ETA}{{\eta}}
\newcommand{\eo}{{\e_0}}
\newcommand\Nb[3]{U_{#1}(#2,#3)}

\newcommand{\m}{{\frak{m}}}
\renewcommand{\M}{{\frak{M}}}

\numberwithin{equation}{section}

\title{Local Minimizers of the Anisotropic Isoperimetric Problem on Closed Manifolds}

\author[A. De Rosa]{Antonio De Rosa}
\address{A. De Rosa: \parbox{\linewidth}{Department of Decision Sciences and BIDSA, Bocconi University, Milano, Italy,\\
Department of Mathematics, University of Maryland, College Park, USA}}
\email{antonio.derosa@unibocconi.it, anderosa@umd.edu}

\author[R. Neumayer]{Robin Neumayer}
\address{R. Neumayer: Department of Mathematical Sciences, Carnegie Mellon University, Pittsburgh, PA 15213, USA}
\email{neumayer@cmu.edu}
\begin{document}
\maketitle

\begin{abstract}
	Local minimizers for the anisotropic isoperimetric problem in the small-volume regime on closed Riemannian manifolds are shown to be geodesically convex and small smooth perturbations of tangent Wulff shapes, quantitatively in terms of the volume.
\end{abstract}
\section{Introduction}
Perimeter-driven variational problems play a central role across analysis and geometry, serving as a tool to investigate the geometry of Riemannian manifolds and constituting the foundation of various models of real-world phenomena in materials science and physics. Describing the shape of local and global minimizers is a central aim in study of these problems.

For the isoperimetric problem on a general closed Riemannian manifold, the qualitative description of the shape of small-volume global minimizers goes back to work of Kleiner as described in \cite{Tomter}; see also \cite{MoJo, RosSurvey, Nardulli, Fall} and references therein among the many works dedicated to this topic. In the realm of physically motivated models, small-volume capillary droplets on a flat surface were shown to be asymptotically spherical in \cite{FinnDrops, TamaniniDrops} (see also \cite{FinnBook}), and the shape of crystalline materials interacting with a convex potential was addressed in \cite{McCannPlanarCrystals, Figalli-Maggi-drops, DPMtwopoint}. In   \cite{Figalli-Maggi-drops}, Figalli and Maggi developed an approach to describing global minimizers in the small-volume regime {\it quantitatively} in terms of the volume, based on the quantitative Euclidean isoperimetric inequality \cite{FuMaPr, FiMaPr, CiLe}. The technique was further developed to prove a quantitative description of small-volume capillary droplets in a container in \cite{MaggiMihaila}. Quantitative isoperimetry has also been crucial in explicitly characterizing small-mass global minimizers for the liquid drop model in nuclear physics \cite{KnMu1, KnMu2, BC}.

Recent years have seen a surge of interest in the shape of {\it local} minimizers, and more generally of critical points, of geometric variational problems. At the forefront of this line of research was work of Ciraolo and Maggi \cite{CM}, where the authors prove and apply a quantitative version of Alexandrov's theorem to show that any volume-constrained local minimizer  of a capillarity-type energy consisting of perimeter plus potential energy with sufficiently small volume is quantitatively close to a ball; see also \cite{KM, BellCapillary, Nonlocal}. For anisotropic surface energies interacting with an external potential in $\R^n$, it was shown in  \cite{DMMN} that small-volume local minimizers that are assumed {\it a priori} to be smooth and to satisfy scale-invariant diameter bounds are quantitatively close to a Wulff shape. Very recently, a qualitative description of small-volume local minimizers of the capillary droplet problem in a container was derived \cite{DelgadinoWeser}, also under a priori smoothness and diameter assumptions. These results are all based on quantitative or qualitative stability of Heintze-Karcher-type inequalities.

In this paper, we use a different approach to the study of local minimizers. We focus on the context of volume-constrained local minimizers of anisotropic surface energies on a closed Riemannian manifold, though we expect that the ideas will apply in other settings.  This approach combines compactness, via achieving scale-invariant diameter bounds, and the rigidity of critical points of the blow-up problem.

Fix a closed Riemannian manifold $(M,g)$ of dimension $n\geq 2$, $\alpha\in (0,1)$, and a $C^{2,\alpha}$ {\it elliptic integrand} $F$ on $M$. More precisely, let $F: T M\to \R$ be a function whose restriction to the unit tangent bundle is $C^{2,\alpha}$ and whose restriction $F_{x_0}(\cdot):=F(x_0, \cdot) :T_{x_0}M\to \R$ is a convex, positively one-homogeneous function with  $F^2_{x_0}$ uniformly convex, $C^{2,\alpha}$, and positive except at the origin for each $x_0 \in M$; see Section~\ref{ssec: integrands}.
The associated {\it anisotropic surface energy} of a set of finite perimeter $\Omega \subset M$ is 
\[
\F(\Omega ) = \int_{\pa^* \Omega} F\big(x, \nu_\Omega(x)\big)\, d\mathcal{H}^{n-1}_g.
\]
A set of finite perimeter $\W$ is said to be a {\it volume-constrained $\eo$-local minimizer of $\F$} if 
$
\F(\W) \leq \F(E)
$ for any competitor $E \subset M$ with 
\[
|E|_g= |\W|_g = v \qquad \text{ and } \qquad E\Delta \W \subset \Nb{g}{\partial \W}{\eo v^{\sfrac{1}{n}}}\,
\]
Here, for any $r>0$ and measurable set $E$ we let $\Nb{g}{E}{r}$ denote the tubular neighborhood
\begin{equation}
	\label{eqn: tube def}
	\Nb{g}{E}{r} = \{ x \in M: d_g(x,y) <r\}\,.
\end{equation}
Volume-constrained $\eo$-local minimality is scaling invariant: if $\W$ is a volume constrained $\eo$-local minimizer of $\F$ in $(M,g)$ with volume $v$, then it is a volume constrained $\eo$-local minimizer of $\F$ in $(M, v^{-\sfrac{2}{n}}g)$ with volume $1$. 
Restricting $F$ at a point $x_0 \in  M$ induces a translation invariant elliptic integrand $F_{x_0}$ and a  corresponding anisotropic surface energy $\bar{\F}_{x_0}$ defined for sets of finite perimeter in $T_{x_0}M$. The surface energy $\bar{\F}_{x_0}$ is minimized for any volume constraint by a translation or dilation of the unit-volume {\it Wulff shape} $K_{x_0}\subset T_{x_0}M$ corresponding to $\bar{\F}_{x_0}$; again see Section~\ref{ssec: integrands}.

We prove that a volume-constrained $\e_0$-local minimizer of $\F$ for sufficiently small volume is a small $C^{2,\alpha}$ perturbation of the image of a Wulff shape $K_{x_0}$ under the exponential map at a point $x_0 \in M$. The estimates are quantitative with respect to the volume constraint.

\begin{theorem}\label{thm: local min}  
	Let $(M,g)$ be a closed Riemannian manifold of dimension $n \geq 2$, fix $\alpha\in (0,1)$, and let $\F$ be an anisotropic surface energy corresponding to a $C^{2,\alpha}$  elliptic integrand $F$. For any  $\eo > 0$ and  $\kappa>0$, there exist $v_0(g, F, \eo, \kappa) \in (0 , |M|)$ and $C=C( g, F, \alpha, \eo,  \kappa)>0$ such that  the following holds. 
	
	Let $\W_v$ be a  volume-constrained $\eo$-local minimizer of $\F$ with volume $v\in ( 0, v_0]$ and $\F(\W_v) \leq \kappa v^{\sfrac{(n-1)}{n}}$. Then $\W_v$ is a geodesically convex $C^{2,\alpha}$-domain. Moreover, 
	there is a point $x_0 \in M$ such that 
	\begin{equation}
		\label{eqn: quant hausdorff0}
	\frac{d_{H,g}\Big(\partial \W_v, \exp_{x_0}(\partial v^{\sfrac{1}{n}}K_{x_0})\Big)}{v^{\sfrac{1}{n}}} < C v^{\sfrac{1}{2n^2}} \qquad \mbox{and} \qquad \frac{\Big|\W_v \Delta \exp_{x_0}(v^{\sfrac{1}{n}}K_{x_0})\Big|_g}{v} < Cv^{\sfrac{1}{2n}} \,. 
	\end{equation}
\end{theorem}

Let us make a few remarks about Theorem~\ref{thm: local min}.

\begin{remark}
{\rm 
The property of $\W_v$ being of class $C^{2,\alpha}$ follows from \eqref{eqn: quant hausdorff0} and the regularity theory for minimizers of anisotropic surface energies (see Section \ref{ssec: reg}). Actually, arguing as in the proof of \cite[Theorem 2]{Figalli-Maggi-drops}, one can strengthen the quantitative Hausdorff estimates \eqref{eqn: quant hausdorff0} and prove that $\partial \W_v$ is locally a $C^{2,\alpha}$ graph, with quantitative estimates in $v$. By Schauder estimates, if $F$ is $C^{k,\alpha}$, one can upgrade the quantitative regularity of $\W_v$ to $C^{k,\alpha}$. { We observe that the minimal requirement for $F$ being $C^{2,\alpha}$ is needed not only to prove that $\W_v$ is a geodesically convex $C^{2,\alpha}$-domain, but also to prove the closeness to the Wulff shape, as we make use of the Alexandrov-type theorem of the first author, Kolasi\'{n}ski, and Santilli \cite[Corollary 6.8]{DRKS}, which requires $F$ being $C^{2,\alpha}$.}
}
\end{remark}

\begin{remark}\label{rmk: energy bound}
{\rm
The smallness of the volume $v_0$  in Theorem~\ref{thm: local min} depends on the constant $\kappa$ of the anisotropic surface energy bound $\F(\W_v) \leq \kappa v^{\sfrac{(n-1)}{n}}$. {This dependence is not needed in \cite[Theorem 2]{Figalli-Maggi-drops}, as global minimizers in the Euclidean setting automatically satisfy an upper bound on the anisotropic surface energy, by means of competition arguments. On the contrary,} this dependence is crucial at several points in our proof. A posteriori, however, Theorem \ref{thm: local min} implies the {\it improved} bound $\F(\W_v) \leq C v^{\sfrac{(n-1)}{n}}$, where $C=C(F,g)>0$ depends only on the maximum (in $x$) of the anisotropic surface energies of the Wulff shapes $K_{x}$. Hence it is conceivable that one could upgrade Theorem \ref{thm: local min}, removing the assumption on the anisotropic perimeter bound and obtaining constants $v_0$ and $C$ that are uniform in $\kappa$. We did not succeed in doing so and we leave it as an interesting open question. 
}
\end{remark}
\begin{remark}
{\rm 
In the case of {\it global} minimizers of $\F$, one may hope to prove a more refined statement than the one in Theorem~\ref{thm: local min}, namely providing information about the point $x_0\in M$ at which a global minimizer is centered. We expect a global minimizer is centered at a point $x_0$ where $\bar{\F}_{x_0}(K_{x_0})$ is minimized. This question becomes intriguing  if one considers an integrand $F$ such that $\bar{\F}_{x_0}(K_{x_0})$ is constant. Here we expect global minimizers to be centered at a point maximizing a weighted average of the sectional curvatures of $(M,g)$ that is compatible with the anisotropy of $F$, playing the role of the scalar curvature for the isoperimetric problem.  We leave this point as an interesting open problem.
}
\end{remark}

\begin{remark}{\rm 
For global minimizers, a slightly weaker version of Theorem \ref{thm: local min}, which is stated in Theorem \ref{thm: uniform convergence2}, can be generalized to every convex integrand that is $C^1$ in the $x$ variable. Indeed, we just need to replace the use of the Alexandrov-type theorem of the first author, Kolasi\'{n}ski, and Santilli \cite[Corollary 6.8]{DRKS} in the proof of the intermeditate Theorem \ref{thm: uniform convergence} with the uniqueness (up to translations) of the Wulff shape among global minimizers \cite{bams,fonsecam1991,bromo,milman}. The rest of the proof of Theorem \ref{thm: uniform convergence2} remains unchanged.
}
\end{remark}

Let us discuss the proof of Theorem~\ref{thm: local min}. First,  a {qualitative} form of Theorem~\ref{thm: local min} requires three key ingredients: (1) a diameter bound, (2) a compactness argument, and (3) a classification of local minimizers in the blow-up limit.  Ingredient (3) was proven in \cite{DRKS}, where the first author, Kolasi{\'n}ski, and Santilli showed that Wulff shapes are the only critical points of the blowup problem among sets of finite perimeter.  The main contributions of this paper are ingredients (1) and (2). 

For ingredient (1), in the context of the standard {perimeter functional},  a sequence of volume-constrained local minimizers with a scale-invariant perimeter bound and  volume $v_k \to 0$ can be shown to have uniformly bounded (constant) mean curvature with respect to the rescaled metrics $h=v_k^{\sfrac{-2}{n}}g$  via the Heitze-Karcher inequality  (see \cite[Theorem 2.2]{MoJo} for global minimizers).  The {\it area monotonicity formula} then implies a uniform diameter bound.  For any anisotropic surface energy that is not an affine transformation of the perimeter,  however,  no monotonicity formula is available \cite{allard1974characterization}!  

Establishing uniform density estimates is another known technique for achieving diameter bounds, but it also falls short in this setting.  An argument of Almgren \cite{Almgren} (see also \cite{GonMasTam83,MorganIsop}) shows that  $\W_v$ is a quasi-minimizer of $\F$ and
  satisfies density estimates, but with constants that depend on the set itself and in particular are not uniform in $v$ in a scale invariant sense. {Uniform} quasi-minimality and thus uniform density estimates  can be achieved by scaling for some volume-constrained problems in Euclidean space, see e.g. \cite{FuJu, N16}, but this technique is clearly specific to $\R^n$. ({\it After} proving the diameter bound, we apply this technique in charts to prove in   Lemma~\ref{lem: volume constraint} that local minimizers as in Theorem~\ref{thm: local min} satisfy uniform density estimates in a scale invariant sense.)

 Another approach to obtaining a diameter bound was shown in \cite{Figalli-Maggi-log} and \cite{MaggiMihaila}, for global minimizers of weighted Euclidean isoperimetric problems and capillary drops in a container in $\R^n$ respectively.  
 Here the idea is to partition a global minimizer $E_v$ with a well-chosen collection of cubes $\{Q_i\}_{i=1}^N$.  Applying the isoperimetric inequality to each element of the partition they bound from below the sum of the energies of the partition elements. On the other hand, by global minimality, this sum can be bounded from above by the same constant, obtaining an estimate of the following type:
 \begin{equation}
 	\label{eqn: sandwich}
0 \leq \sum  \bigg( \frac{|\Omega_v \cap Q_i|_g}{v}\bigg)^{\sfrac{(n-1)}{n}} -1 \leq \epsilon(v)
  \end{equation}
The concavity of the function $t \mapsto t^{\sfrac{(n-1)}{n}}$ immediately implies that for any $1 \leq L \leq N$, 
\[
   \bigg( \sum_{i=1}^L \frac{|\Omega_v \cap Q_i|_g}{v}\bigg)^{\sfrac{(n-1)}{n}} +  \bigg(\sum_{i=L+1}^N \frac{|\Omega_v \cap Q_i|_g}{v}\bigg)^{\sfrac{(n-1)}{n}} -1 \leq \epsilon(v),
 \]
from which one deduces that $|\Omega_v \cap Q_i| \geq 1-\e$ for some $i$ and for $v$ small enough. Then a classical use of coarea formula provides a standard differential inequality which allows to prove a diameter bound. We remark that in doing so, in order to make $\epsilon(v)$ small enough, the radius of the cubes should be optimized at a scale $v^{1/2n}$. However this provides a suboptimal diameter bound that in the scale invariant sense blows up as $v\to 0$.

This sandwiching argument relies on having precisely the constant $1$ in \eqref{eqn: sandwich}. Underpinning this is the fact that the energy of a small-volume global minimizer is asymptotically equal to the isoperimetric constant of the blow up problem. Thus, this approach is confined to the setting of global minimizers.

For the aforementioned reasons, we need to take a new approach to obtain our scale invariant diameter bound for volume-constrained $\eo$-local minimizers of $\F$ with small volume (Section \ref{sec:diam}). It is possible to adapt the ideas of  \cite{Figalli-Maggi-log, MaggiMihaila} to obtain a similar partition of a local minimizer $\W_v$ on the compact manifold. However, as explained above, one cannot obtain the constant $1$ in \eqref{eqn: sandwich} (or equivalently, make the right-hand side of \eqref{eqn: sandwich} small). To overcome this problem we use a general concavity lemma (Lemma \ref{lem: J balls}) for sequences of real numbers, which has been previously utilized in concentration compactness arguments for isoperimetric type problems; see for instance  \cite{Gol22,candautilh2021existence,CGOS, novaga2022isoperimetric}. 
 We apply this lemma to show that most of the volume of $\Omega_v$ is contained in the union of $\JO$ balls of radius $v^{1/n}$, where $\JO$ depends just on $n$, $F$, and the anisotropic isoperimetric ratio of the finite perimeter set.
From this point, we can use $\W_v$ intersected with the $\JO$ balls as a competitor, to deduce a standard differential inequality which allows us to obtain the  diameter bound.
Beyond the fact that this argument works for {\it local } minimizers, this approach provides a diameter bound that does not blow up in a scale invariant sense.  

%
The diameter bound is the starting point for ingredient (2), allowing us in Section \ref{sec: qual} to pull back and rescale a sequence of local minimizers in ($J_0$) charts and obtain $L^1$ convergence to a set $E$. Using the rigidity theorem \cite[Corollary 6.8]{DRKS} (Ingredient (3)), we deduce the limiting set is a {\it translation} of a tangential Wulff shape. This translation, whose modulus a priori could be much larger than the natural length scale $v^{1/n}$, leads to serious difficulties in transmitting this information about the blowup back to the local minimizers on the manifold. Hence we need to compare the shape of the limiting translated Wulff shape and a tangent Wulff shape at a different appropriately chosen point. 
Since the integrand $F$ is not autonomous, that is, it is $x$-dependent,  careful analysis is needed  
to carry this out in Section~\ref{sec: app}.

The final step (Section \ref{sec: quantitative}) is to provide a quantitative version of the closeness of the previous step, meaning that the closeness is not only scale invariant, but will actually decay quantitatively as a power of the volume. The key ingredient for this is the quantitative Wulff inequality of Figalli, Maggi, and Pratelli \cite{FiMaPr}.

\subsection*{Acknowledgments}
	 Antonio De Rosa was partially supported by the NSF DMS CAREER Award No.~2143124 and by the European Union: the European Research Council (ERC), through StG ``ANGEVA'', project number: 101076411. Views and opinions expressed are however those of the authors only and do not necessarily reflect those of the European Union or the European Research Council. Neither the European Union nor the granting authority can be held responsible for them.Robin Neumayer is partially supported by NSF Grant DMS-2155054 and the Gregg Zeitlin Early Career Professorship. Both authors are indebted to Michael Goldman for  showing us the current much simpler proof of Lemma \ref{lem: J balls} and sharing with us its use in the setting of concentration compactness arguments. Both authors warmly thank Nick Edelen for a useful discussion.

\section{Preliminaries}
In this section we introduce definitions and notation  and prove some preliminary results that will be needed in the remainder of the paper.
\subsection{Basic Notation}
Consider a smooth Riemannian manifold $(M,g)$ of dimension $n\geq 2.$ Let $B_g(x,r) = \{ y \in M ; d_g(x,y) <r\}$ denote the geodesic ball of radius $r>0$ centered at $x \in M$. Recall the notation  $\Nb{g}{E}{r}= \{ y \in M : d_g(x, E)<r\}$ introduced in \eqref{eqn: tube def} for the tubular neighborhood of a set $E \subset M$. The Hausdorff distance between sets $\Sigma, \Sigma' \subset M$ is defined by 
$$d_{H,g} (\Sigma, \Sigma' ) = \inf\{ r>0 : \Sigma \subset \Nb{g}{\Sigma'}{r} \text{ and } \Sigma'\subset \Nb{g}{\Sigma}{r}\}.$$ 
For the $k$ dimensional Hausdorff measure with respect to the metric $g$, we write $\Hi^k_g(\cdot ) $. When $k=n$ we simply write $|\cdot |_g = \Hi^n_g(\cdot)$ since $\Hi^n_g$ agrees with the standard volume measure induced by $g$. We call a set measurable if it is $\Hi^n_g$-measurable. It is worth noting how these quantities behave under rescaling the metric $g$: if $h=r^{-2} g$ for $r>0$, then for $x \in M, E\subset M,$ and $\rho >0$ we have 
\[
B_h(x,\rho) = B_g( x, r\rho),\quad  \Nb{h}{E}{\rho} = \Nb{g}{E}{r\rho},\quad d_{H,h} (\Sigma, \Sigma' ) =\frac{d_{H,g} (\Sigma, \Sigma' )}{r} , \quad \Hi^k_h(E) = \frac{\Hi^k_g(E)}{r^{k}}.
\]
We denote the Euclidean metric by $g_{euc}$. 

We let $\text{inj}_gM>0$ denote the injectivity radius of $M$, i.e. the supremum over $r>0$ such that the exponential map $\exp_x: T_x M \to M$ is a diffeomorphism from $B_{g_x}(0,r)$ to $B_g(x,r)$ for all $x \in M$. 

Let $E\Delta G = (E\setminus G) \cup (G\setminus E)$ be the symmetric difference between sets, which we note satisfies the triangle inequality-type property 
\begin{equation}
	\label{eqn: triangle}
	E\Delta G \subset (E\Delta E') \cup (E'\Delta G).
\end{equation}
A sequence of measurable sets $\{ E_i\}$ converges in $L^1$ to $E$ if $|E_i\Delta E|_g \to 0$. 

\subsection{Sets of finite perimeter}\label{ssec: sofp}
 We work in  the framework of sets of finite perimeter. A measurable set $E\subset M$ is a set of finite perimeter if
\[
P_g(E)=\sup\left\{ \int_E \text{div}_g T(x)\,d\Hi^n_g(x)  : T \in \mathcal{X}_c(M), |T(x)|_g \leq 1 \text{ for all } x \in M \right\}< +\infty.
\]
Here $\mathcal{X}_c(M)$ denotes the space of smooth compactly supported vector fields on $M$. { Although many results we will quote from the theory of finite perimeter sets are proved in the Euclidean space, they can be easily extended to Riemannian manifolds, see for instance \cite[Chapter 2.4]{UsefulThesis}.}

For any set of finite perimeter $E\subset M$, by the Riesz representation theorem for bounded linear functionals on $\mathcal{X}_c(M)$ (\cite[Section 1.2]{MPPP}, \cite[Theorem 2.36]{UsefulThesis}),
 there is a finite Radon measure $|D1_E|$ on $M$ and a $|D1_E|$-measurable vector field $\nu_E :M \to TM$ with $|\nu_E(x)|_g =1$ for $|D1_E|$-a.e. $x \in M$, such that the distributional gradient $D1_E$ has the representation $D1_E = \nu_E |D1_E|$ (in other words,  $\int_E \text{div}_g T \, d\mathcal{H}^n = \int_M \langle T, \nu_E\rangle_g \, d |D1_E|$ for any $T \in \mathcal{X}_c(M)$). Note that $\nu_E$ depends on the metric $g$, though we suppress this dependence in the notation when there is no ambiguity.
 
  The reduced boundary $\partial^*E$ of $E$ is defined by
\[
\partial^*E = \{ x \in \text{spt} |D1_E|\ : \ |\nu_E(x)|_g =1\}.
\] 
It is easy to show that $\text{spt}|D1_E|$ and thus $\partial^*E$ are contained in the topological boundary $\partial E$. Note that $D1_E = \nu_E |D1_E|\llcorner \partial^*E$.  By the De Giorgi Structure Theorem (see \cite[Theorem 15.9]{Mag} in the Euclidean case; the proof can be adapted to the setting of Riemannian manifolds), $\partial^*E$ is an $\mathcal{H}^{n-1}$-rectifiable set and $D1_E = \nu_E \mathcal{H}_g^{n-1}\llcorner \partial^*E$. 

We denote with $E^{(1)}$ the set of points of density $1$ for $E$, i.e.
$$E^{(1)}:=\left\{x\,:\, \lim_{r\to 0}\frac{|B_g( x, r)\cap E|_g}{|B_g( x, r)|_g}=1\right\}.$$

Sets of finite perimeter enjoy a useful compactness property with respect to the $L^1$ topology: if $\{E_i\}$ is a sequence of sets of finite perimeter in $M$ with $E_i \subset A$ for a compact set $A$ and $\sup_i P(E_i) <+\infty$, then up to a subsequence, $E_i \to E$ in $L^1$ for a set of finite perimeter $E\subset A.$ 

All of these properties of sets of finite perimeter are invariant by modification of the set on an $\mathcal{H}^n_g$-negligible set. By \cite[Prop. 12.19]{Mag}, we can modify $E$ on a $\mathcal{H}^n_g$-negligible set to ensure that 
\begin{equation}
	\label{eqn: cleaned up sets}
\overline{\partial^*E} =\left\{x \in M : 0< |E\cap B_g(x,r)|_g <|B_g(x,r)|_g \text{ for all }r>0 \right\} = \pa E\,.
\end{equation}
(The first identity always holds, while the second holds after a measure zero modification.) In the sequel we will tacitly assume that every set of finite perimeter has been cleaned up in this way.

\subsection{Elliptic integrands}\label{ssec: integrands}

Let $F : TM \to \R$ be an anisotropic integrand on $(M,g)$ as defined in the introduction.  Given a diffeomorphism  $\psi: U \to V \subset M$,  define the pulled-back integrand 
 \begin{equation}
 	\label{eqn: pull back}
\psi^*F : TU \to \R \qquad \text{  by } \qquad \psi^*F(x , \nu ) = F(\psi(x) , d\psi_x \nu).
 \end{equation}
When no confusion can arise, we will write $\psi^*F= F^*$; this should not be confused with the dual integrand $F_*$ used in Appendix~\ref{sec: app}. 
Let $r_0= \text{inj}_g(M)/2$. The $C^{2,\alpha}$ regularity of $F$ in particular implies that for any $x_0 \in M$ if we take $\psi:B_{g_{euc}}(0 r_0) \to B_{g}(x_0, r_0)$ to be the normal coordinate map, then 
\begin{equation}
	\label{eqn: F C1 reg}
 \sup\left\{ \| \psi^*F(\ \cdot\  , \nu)\|_{C^{1}(B_{g_{euc}}(0,r_0))} \ :   \ \nu \in S^{n-1} \right\} \leq C(n,g,F).
\end{equation}
For the majority of the paper we will only use this $C^1$ regularity of $F$ in $x$ and will only use the higher regularity in $x$ when applying the $\e$-regularity theorem and Schauder estimates in Section~\ref{sec: quantitative}.

For $x_0 \in M$, the restricted integrand $F_{x_0} :T_{x_0} M \to \R $ defined by $F_{x_0}(\nu) = F(x_0,\nu)$ gives rise to a translation invariant surface energy for sets of finite perimeter $E \subset T_{x_0}M$:
\[
\bar{\F}_{x_0}(E) = \int_{\partial^* E} F_{x_0}(\nu_E(x)) \, d\Hi^{n-1}_{g_{x_0}}(x).
\]
Here the Hausdorff measure and measure theoretic outer unit normal are taken with respect to the metric $g_{x_0}( \cdot ,\cdot)$ on $T_{x_0}M$. 
Among sets of a fixed volume in $T_{x_0}M$, the energy $\bar{\F}_{x_0}$ is uniquely minimized by translations and dilations of the volume-$1$ {\it tangent Wulff shape} $K_{x_0} \subset T_{x_0}M$ defined by
\begin{equation}
	\label{eqn: tangent wulff A}
K_{x_0} := \frac{\hat{K}_{x_0}}{|\hat{K}_{x_0}|_{g_{x_0}}} \qquad \mbox{where $\hat{K}_{x_0}:= \big\{ y \in T_{x_0}M\  : \ g_{x_0}( y ,\nu) < F_{x_0}(\nu) \text{ for all } \nu \in T_{x_0}M\setminus\{0\}\big\}$\,;}
\end{equation}
see \cite{bams,fonsecam1991,bromo,milman}. Note that $0 \in K_{x_0}$.  This minimality property is stated in scale-invariant form as the {\it Wulff inequality}:
\begin{equation}\label{Wulffineq}
\bar{\F}_{x_0}(E) \geq n|E|_{g_{x_0}}^{\frac{n-1}{n}}|\hat{K}_{x_0}|_{g_{x_0}}^{1/n}=\left(\frac{|E|_{g_{x_0}}}{|\hat{K}_{x_0}|_{g_{x_0}}}\right)^{\frac{n-1}{n}}\bar{\F}_{x_0}(\hat{K}_{x_0})=|E|_{g_{x_0}}^{\frac{n-1}{n}}\bar{\F}_{x_0}(K_{x_0})
\end{equation}
We set
\begin{equation}\label{eqn:massimo}
\C = \sup\{ P_{g_{x}}(K_x ) : x \in M\}.
\end{equation}

Define the quantities	$\M = \sup\{ F(x,\nu ) : x \in M , |\nu|_g= 1\}$ and $\m = \inf\{ F(x,\nu)  : x \in M , |\nu|_g= 1\}$, which are positive and finite by the assumptions on $F$ and the compactness of $(M,g)$. Moreover,  $\m P_g(E) \leq \F(E) \leq \M P_g(E)$ for any set of finite perimeter $E\subset M$ and $B_{g_{x_0}}(0, \m) \subset K_{x_0}\subset B_{g_{x_0}}(0, \M)$.
	  Estimating the isoperimetric profile above by taking geodesic ball competitors, we thus find that there exists $\bar{v}=\bar{v}(n,g)\in (0,|M|_g)$ such that 
\begin{equation}\label{eqn: Wulff small volume}
	\F(E) \geq \,{\frac{\m}{2}\, n \omega_n^{\sfrac{1}{n}}} |E|_g^{\sfrac{(n-1)}{n}} \qquad \text{ for }E\subset M\  \text{ with }\ |E|_g\leq \bar{v}.
\end{equation}

It is also useful to notice how the surface energy behaves under rescaling the metric. Setting $h= r^{-2}g$ for $r>0$, the same integrand defines an anisotropic surface energy
\[
\F_h(E) = \int_{\partial^*E}F(x,\nu_E^h(x)) \, d\Hi^{n-1}_h(x) = \frac{ \F(E)}{r^{n-1}}. 
\]
\subsection{Classical regularity results for local minimizers}\label{ssec: reg}
Let $\W_v$ be  a volume-constrained $\eo$-local minimizer of $\F$
 with volume $v \in (0,|M|)$.  A classical argument dating back to Almgren \cite{Almgren} (see also \cite{GonMasTam83,MorganIsop}) shows that  $\W_v$ is a quasi-minimizer of $\F$ and 
  satisfies density estimates. 
  Unlike the {\it uniform} quasi-minimality and density estimates we will ultimately derive in  Lemma~\ref{lem: volume constraint}, these estimates depend  on the set $\W_v$ itself, and thus are not be directly useful to proving Theorem~\ref{thm: local min}. However, they show that $\Hi^n$-a.e. $x \in \partial E$ is a point of density strictly between $0$ and $1$; together with Federer's theorem \cite[Theorem 16.2]{Mag} this implies that $\Hi^n(\partial \W_v\setminus \partial^*\W_v) =0$ and thus, up to modifying $\W_v$ by an $\Hi^n$-negligible set, we can  replace a local minimizer with an open set representative. In the remainder of the paper we will always take this representative.

In fact, the classical regularity theory, \cite{Almgren1, Bomb, SSA, SchoenSimon, DS} shows that $\partial \W_v$ is a $C^{1,\alpha}$ hypersurface outside of a singular set of $\mathcal{H}^{n-2}$ measure zero. These estimates are again not a priori uniform in $v$.

\subsection{Matching the volume constraint}

At several points throughout this paper, we will have a set whose volume is close, but not exactly equal, to a certain prescribed volume. To use it as a competitor for volume-constrained local minimality, we need to replace it with a set that exactly satisfies the volume constraint. The following technical lemma allows us to do this in such a way that the difference between the surface energies of the original set and the modified set is quantitatively controlled in terms of the volume error.  For the proof we borrow some ideas from \cite[Lemma 3.1]{MaggiMihaila}.

\begin{lemma} \label{lemma: dilation}
For any $n\geq 2$, there exist $\ETA_0>0$ and $c_n>0$ depending only on $n$ such that the following holds. 
Fix $D >1$, $\ETA\in(0,\ETA_0)$ and let $(N,h)$ be a closed Riemannian $n$-manifold with $\text{inj}_h N >D$ and  
\begin{equation}\label{eqn: metrics close 2}
	\| \psi^*h - g_{euc}\|_{C^1(B_{g_{euc}}(0,D))} \leq \frac\ETA D
\end{equation}
in normal coordinates $\psi$ centered at any $x \in M$.
 Let $E\subset M$ be a measurable set with $||E|_h -1|<\ETA$ for which there exist finitely many disjoint open sets $\{V_i\}_{i=1}^K$ with $\text{diam}_h(V_i) \leq D$ such that $\Nb{h}{E}{1} \subset \cup_{i=1}^K V_i$. Then there is a measurable set $\tilde E \subset \cup_{i=1}^K V_i$  such that 
\begin{align}\label{eqn: 2 nbhd}
	\Emod \Delta E \Subset \Nb{h}{\pa E}{c_n D\ETA }\qquad \text{ and }\qquad	|\Emod|_h=1.
\end{align}
Moreover, if $\F$ is an anisotropic perimeter with integrand $F$ satisfying
\begin{equation}
	\label{eqn: aniso bound}
\| \psi^*F(\cdot ,\nu) - \psi^*F(0, \nu)\|_{C^1(B_{g_{euc}}(0,D))} \leq 1
\end{equation}
for every $x$ and $\nu$, then 
\begin{equation}
	\label{eqn: aniso estimate}
\F(\Emod) \leq \F(E)\left(1 + c_n D \left||E|_h -1\right|\right).
\end{equation}
\end{lemma}
\begin{proof}
Fix $\ETA_0>0$ to be specified later in the proof and let $\ETA\in (0,\ETA_0)$ and $D >1$.
 For each $k \in 1,\dots, K$, 
since $\text{diam}_hV_k \leq D,$ we can find $x_k \in V_k$ with $V_k\subset B_{g_{euc}}(x_k, D).$
 Let $\psi_k: B_{g_{euc}}(0,D) \to  B_h(x_k, D)$ be the normal coordinate map at $x_k$. 
 For $\lambda \in [\frac1 2,\frac3 2],$ define the set $E_\lambda\subset M$ by
$$  E_{\lambda}:= \mlcup_{k=1}^K \psi_k\left( \lambda\, \psi_k^{-1}\big( V_k \cap E\big)\right).$$
Note that $E_1 = E$ and that $E_\lambda$ is well-defined for the full interval since we have assumed  $\Nb{h}{ E\cap V_k}{1}\subset V_k$.
Moreover, letting $\lambda_+ =(1-\ETA )^{-(n+1)/n}$ and $\lambda_-=(1-\ETA )^{-(n+1)/n}$ and recalling  $(1-\ETA)\leq |E|_h\leq (1+\ETA)$ by assumption, we compute in coordinates to find
 \begin{align*}
 |E_{\lambda_+}|_h& \geq \lambda_+^n(1-\ETA)^{n}|E|_h\geq \lambda_+^n(1-\ETA)^{n+1}\geq 1, \text{ and }\\
  |E_{\lambda_-}|_h & \leq \lambda_-^n(1+\ETA)^{n}|E|_h\leq \lambda_-^n(1+\ETA)^{n+1}\leq 1. 	
 \end{align*}
The function $\lambda\mapsto |E_{\lambda}|_h$ is continuous, and thus we can find
 we can find $\lambda_0  \in [\lambda_-,\lambda_+] \subset [1-c_n\ETA,1+c_n\ETA]$ such that $|E_{\lambda_0}|_h=1$. We set $\Emod = E_{\lambda_0}$. 
 
 In order to prove the remaining properties of $\Emod,$ we prove the following key fact:
 	\begin{equation}
 		\label{eqn: lambda bound2}
	|\lambda_0-1| \leq c_n||E|_h -1|
\end{equation}	
provided $\ETA_0>0$ is small enough in terms of $n$.  
To this end, we let $\mathcal{J}(x) = \sqrt{\det h_{ij}(x)}$ be the volume form of $h$ in the coordinates defined by $\psi$ and let $G_k = \psi_k^{-1}(V_k \cap E)\subset \R^n$. We therefore have 
	\begin{align}\label{eqn: area2}
		|E_{\lambda}|_h- |E|_h 
		= \sum_{k=1}^K \left\{ \int_{\lambda G_k} \mathcal{J}(x)  \,dx  -\int_{G_k} \mathcal{J}(x)  \,dx \right\}=\sum_{k=1}^K \int_{G_k}\left\{  \lambda^n \mathcal{J}(\lambda x) -  \mathcal{J}(x) \right\}\,dx.
	\end{align}
When $\lambda \in [1, 1 +c_n \ETA]$, 	we add and subtract $\lambda^n \mathcal{J}(x) $ and use the fundamental theorem of calculus to find
	\begin{align*}
	 \lambda^n \mathcal{J}(\lambda x)  - \mathcal{J}(x) &\geq  \left(\lambda^n -1\right)  \, \mathcal{J}(x) - \lambda^n \left(\mathcal{J}(\lambda x) -\mathcal{J}(x)\right)\\
	& \geq (\lambda-1)\, \mathcal{J}(x)  - (\lambda-1)\sup\left\{|\nabla \mathcal{J}(tx)\cdot x| : t \in [1-c_n\ETA, 1+c_n\ETA]\right\}
	\end{align*}
	for any fixed $x$.
	Since  \eqref{eqn: metrics close 2} guarantees that $\mathcal{J}(x) \geq 1-\ETA$ and that the term in brackets is bounded above by $1/2$ for all $x \in B(0,D)$ provided $\ETA_0>0$ is small enough depending on $n$, we see that $\lambda^n \mathcal{J}(\lambda x)  - \mathcal{J}(x) \geq c_n (\lambda -1) \mathcal{J}(x)$.  Applying this inequality to the right-hand side of \eqref{eqn: area2} we find that \eqref{eqn: lambda bound2} holds in this case. The analogous argument in the case when $\lambda \in [1-c_n\ETA, 1]$ shows that $\mathcal{J}(x)- \lambda^n \mathcal{J}(\lambda x)  \geq  (1-\lambda)c_n \mathcal{J}(x),$	which applied to \eqref{eqn: area2} implies \eqref{eqn: lambda bound2} in this case as well.

Next, with  \eqref{eqn: lambda bound2} in hand, we show the containment
\begin{equation}
	\label{eqn: contain E}
\Emod \Delta E \subset \Nb{h}{\pa E}{c_nD\ETA}\,.
\end{equation}
Indeed,  if $x \in E \setminus E_{\lambda_0}$, then for some $k \in \{1,\dots, K\}$, we have $y:= \psi_k^{-1}(x) \in G_k$ but $y\not\in \lambda_0 G_k$. 
Here we again let $G_k = \psi_k^{-1}(V_k \cap E)\subset\R^n$.  So,  $ty \in \partial G_k$ for some $t \in [1,\lambda_0]$ or $t \in [\lambda_0,1]$, and thus by \eqref{eqn: lambda bound2},  
$$d_{g_{euc}}(y, \partial G_k) \leq |t-1||y| \leq |\lambda_0 -1||y| \leq c_n D \ETA.$$
 Hence by \eqref{eqn: metrics close 2}, we have $d_h(x,\partial E) \leq c_n \ETA D$ (up to doubling $c_n$).  The analogous argument holds for any $x \in E_{\lambda_0} \setminus E$, and thus \eqref{eqn: contain E} holds.

Finally, we show the estimate for the anisotropic perimeter. Similarly to \eqref{eqn: area2}, we have
 \begin{equation}
	\label{eqn: F exp}
	\begin{split}
		\F(E_{\lambda}) - \F(E) 
		& = \sum_{k=1}^K\left\{ \int_{ \partial^* \lambda G_k }\hat{F}(x, \nu_{G_k}(x))  \,d \mathcal{H}^{n-1}_{g_{euc}}
	 -  \int_{\partial^* G_k} \hat{F}(x, \nu_{G_k}(x)) \,d \mathcal{H}^{n-1}_{g_{euc}}\right\}\\
		&	= \sum_{k=1}^K\left\{ \int_{ \partial^* G_k }\lambda^{n-1} \hat{F}(\lambda x, \nu_{G_k}(x))  - \hat{F}(x, \nu_{G_k}(x)) \,d \mathcal{H}^{n-1}_{g_{euc}}\right\}
	\end{split}
\end{equation}
where here we define the function $\hat{F}(x, \nu) = F^*(x,\nu)\sqrt{\det \tilde{h}_{ij}(x)}$ where $\tilde{h}_{ij}$ are the coefficients of the metric on $\pa^* E$ induced by $h$.
			Adding and subtracting terms and using $\lambda \in [1-c_n, 1 +c_n \ETA]$, \eqref{eqn: metrics close 2}, and \eqref{eqn: aniso bound}, a Taylor expansion shows that for any fixed $x$ and $\nu \in S^{n-1}$ that
	\begin{align*}
	\left| \lambda^n \hat{F}(\lambda x,\nu )  - \hat{F}(x,\nu ) \right| 
	&\leq  \left|\lambda^{n-1} -1 \right| \, \hat{F}(x,\nu ) + \lambda^{n-1}\sqrt{\det \tilde{h}_{ij}(\lambda x)}  \left| F^*(\lambda x, \nu ) -F^*(x, \nu )\right|\\
	& + \lambda^{n-1} F^*(x,\nu) \left|\sqrt{\det \tilde{h}_{ij}(\lambda x)} -\sqrt{\det \tilde{h}_{ij}(x)}\right| \leq c_n D(\lambda -1)\hat{F}(x, \nu),
	\end{align*}
 provided $\ETA>0$ is chosen sufficiently small depending on $n$. Combining this with \eqref{eqn: F exp} and recalling \eqref{eqn: lambda bound2}, we conclude that the estimate \eqref{eqn: aniso estimate} holds. This completes the proof.
\end{proof}

\section{Comparing projections of tangent Wulff shapes}\label{sec: app}

The following proposition  compares projections of translated tangent Wulff shapes via the exponential map at different points. 
\begin{proposition}\label{prop: appendix tangent wulff}
	There exist $C = C(g, F)$ and $\rho_0 = \rho_0 (g, F)>0$ such that the following holds. Fix $\rho \in [0,\rho_0)$, choose $x_0, x_1 \in M$ with $d_g(x_0, x_1)<\rho$, and let $K_{x_0}\subset T_{x_0}M$ and $K_{x_1} \subset T_{x_1} M$ be defined as in \eqref{eqn: tangent wulff A}. 
	Then for any $0<r<\rho_0 $, letting $z_1= \exp^{-1}_{x_0}(x_1)$, we have 
	\begin{align*}
		d_{H,g} \big(\exp_{x_1} (\partial r K_{x_1} ) ,  \ \exp_{x_0} ( \partial r K_{x_0} + z_1) \big)& < C\rho\, r,\\ 
		d_{H,g} \big(\exp_{x_1} (r  K_{x_1} ), \ \exp_{x_0} ( r K_{x_0} + z_1) \big) &< C \rho \, r. 
	\end{align*}
\end{proposition}

This section is dedicated to proving Proposition~\ref{prop: appendix tangent wulff}. As in the remainder of the paper we assume that $F$ is a $C^{2,\alpha}$ elliptic integrand. However, we note that the proof of Proposition~\ref{prop: appendix tangent wulff} only requires that $F$ has $C^1$ dependence on $x$ and for each $x_0\in M$, $F(x_0,\cdot)$ is a convex one-homogenous function that is positive expect a the origin, with no smoothness or ellipticity needed.

Together with the metric $g$, the integrand $F$ induces a dual integrand $F_*:TM\to \R$ via
\[
F_*(x,z) = \sup\big\{ g_x(z,\nu) \, : \, \nu \in T_xM, \ F(x,\nu) \leq 1\big\}.
\] 
Given a pair of points $x_0, x_1 \in M$,  let
\[
d_{F_*}(x_0, x_1) = \inf \left\{\int_0^1 F_*(\gamma(t) , \dot 
\gamma(t))  \ : \ \gamma :[0,1]\to M, \ \gamma(0)=x_0,\ \gamma(1) = x_1 \right\}.
\]
If  $F$ is not symmetric, we may have $d_{F_*}(x_0, x_1) \neq d_{F_*}(x_1, x_0)$. Nonetheless we may consider the $F_*$-balls 
$$B_{F_*} (x_0, \rho) = \{ x_1 \in M : d_{F_*}(x_0 , x_1)< \rho\}\,.
$$ 
Recall the quantities $\m$ and $\M$ defined in Section~\ref{ssec: integrands} and note that
$\m^{-1} = \sup \{ F_*(x,z) : x \in M , g_x(z,z) =1\}$ and $\M^{-1} = \inf \{ F_*(x,z) : x \in M , g_x(z,z) =1\}$
and therefore 
\begin{equation}
	\label{eqn: compare g, Fstar}
	\begin{split}
		\M^{-1} d_g(x_0, x_1) &\leq d_{F_*}(x_0, x_1) \leq \m^{-1} d_g(x_0, x_1), \\
		  B_g(x_0, \m \rho) &
\subset B_{F_*}(x_0,\rho) \subset B_g(x_0, \M \rho)\,.
	\end{split}
\end{equation}
Proposition~\ref{prop: appendix tangent wulff} will follow from the next lemma and the triangle inequality.
\begin{lemma}\label{lem: appendix hausdorff}
	There exist $C= C(g,F)>0$ and $\rho_0=\rho_0(g,F)>0$ such that the following holds. Fix $x_0, x_1 \in M$ with $d_g(x_0,x_1) \leq \rho \leq  \rho_0$. Then for all $0< r\leq \rho_0$,  letting $z_1 = \exp_{x_0}^{-1}(x_1),$ we have 
	\begin{align*}
		d_{H,g}\big(\exp_{x_0}\big(\partial r K_{x_0}+z_1\big) ,\ \partial B_{F_*}\big(x_1,r\big)\big) &\leq C \rho \, r,\\
		d_{H,g}\big(\exp_{x_0}\big(r K_{x_0}+z_1\big) ,\  B_{F_*}\big(x_1,r\big)\big) &\leq C \rho\, r.
	\end{align*}
\end{lemma}
Before proving Lemma~\ref{lem: appendix hausdorff}, let us see how it implies Proposition~\ref{prop: appendix tangent wulff}. 
\begin{proof}[Proof of Proposition~\ref{prop: appendix tangent wulff}]
Let  $\rho_0>0$ be chosen according to Lemma~\ref{lem: appendix hausdorff}. Fix $r, \rho\leq \rho_0$ and choose $x_0, x_1 \in M$ with $d_g(x_0, x_1) \leq \rho$. First, we apply Lemma~\ref{lem: appendix hausdorff} to find 
\begin{align*}
	d_{H,g}(\exp_{x_0}\big(\partial \rho {K}_{x_0} + z) ,\ \partial B_{F_*}(x_1 , \rho)\big) \leq C \rho\, r,\\
	d_{H,g} \big(\exp_{x_0} (\rho {K}_{x_0} +z) ,\  B_{F_*}(x_1 ,\rho) \big) \leq  C \rho\, r.
\end{align*}
Next, apply Lemma~\ref{lem: appendix hausdorff} with the roles of both $x_0$ and $x_1$ played by $x_1$ (and thus $z_1 = 0 \in T_{x_1}M$) to find 
\begin{align*}
	d_{H,g}\big(\exp_{x_1}\big(\partial r{K}_{x_1}\big) ,\ \partial B_{F_*}\big(x_1 , r\big)\big) \leq C \rho\, r,\\
	d_{H,g} \big(\exp_{x_0} \big(r {K}_{x_1}\big) ,\  B_{F_*}\big(x_1 ,r\big) \big) \leq C \rho\,r\,.
	 \end{align*}
We apply the triangle inequality to conclude the proof. 
\end{proof}
We need three preparatory lemmas to prove Lemma~\ref{lem: appendix hausdorff}. The first transfers the assumed regularity on $F$ in the variable $x$ to regularity of $F_*$ in  $x$.  Let $r_0= \text{inj}_g(M)/2$. Fix $x_0 \in M$ and let $\psi = \exp_{x_0}$ and consider the pulled-back  integrand 
$\psi^*F : B_{g_{x_0}}(0,r_0) \times T_{x_0}M \to \R$  defined as in \eqref{eqn: pull back}. 
The  regularity \eqref{eqn: F C1 reg} of $F$ in $x$ implies, in particular, that for all $\nu \in T_{x_0}M$ with $g_{x_0}(\nu, \nu)=1$
and
 $\rho<r_0$ we have 
\begin{equation}
	\label{eqn: regularity on F 2}
	\| \psi^*F(\ \cdot \ ,\nu) - \psi^*F(0,\nu)\|_{C^0(B(0,\rho))} \leq \textsf{C}\, \rho
\end{equation}
for a constant $\textsf{C}= \textsf{C}(F,g)$. Define the pulled-back dual integrand  $\psi^*F_*: B_{g_{x_0}}(0,r_0)\times T_{x_0}M \to \R $ by 
\[
\psi^*F_*(y,z) = F_*(\psi(y), d\psi_y(z)).
\]
\begin{lemma}\label{lem: app Fstar reg}
 There exists $\rho_2= \rho_2(g, F)$ and $C=C(g,F)$ such that the following holds.  Let $F: TM\to\R$ be an elliptic integrand satisfying \eqref{eqn: regularity on F 2}.   Then for each $x_0 \in M$ and $z \in T_{x_0}M$ with $g_{x_0}(z,z)=1$, 
   and $\rho<\rho_2$,
    we have 
	\begin{equation}
		\label{eqn: reg dual norm} 
	\| \psi^*F_*(\, \cdot\, , z) - \psi^*(F_*(0,z))\|_{C_0(B(0,\rho ))} \leq C\,\rho.
	\end{equation}
	Here we let $\psi  =\exp_{x_0} :T_{x_0}M\to M$.
\end{lemma}

\begin{proof}
	Choose $\rho_2 \leq \text{inj}_gM/2$ small enough depending on $g$ such that 
	\begin{equation}\label{eqn: pull back 1}
		(1-\rho) g_{x_0} \leq \psi^* g \leq (1+\rho) g_{x_0} \qquad\text{ in }B_{g_{x_0}}(0,\rho) \text{ for all } \rho<\rho_2.
	\end{equation}
For any $z \in T_{x_0} M$ with $g_{x_0}(z,z)=1$, choose $\nu_z \in T_{x_0}M$ such that $F(x_0, \nu_z) =1$ and 
	\begin{equation}
		\label{eqn: equal}
	g_{x_0}(z,\nu_z)  = F_*(x_0, z)=  \psi^*F_*(0,z) .
	\end{equation}
So, choosing any $y \in  B_{g_{x_0}}(0,\rho)\subset T_{x_0}M$, 
the assumption \eqref{eqn: regularity on F 2} implies that $\psi^*F(y ,\nu_z) \leq 1+\textsf{C} \rho$ with $\textsf{C}$ as in \eqref{eqn: regularity on F 2}.
So, using $\overline{\nu}_z := {\nu}_z/\psi^*F(y,\nu_z)$ as a competitor in the definition of $\psi^*F_*(y, z)$, we have  
\begin{align*}
	\psi^*F_*(y, z)= \sup\big\{ (\psi^*g)_y (z,\nu)  : \psi^*F (y, \nu_z) \leq 1\big\} &\geq (\psi^*g)_y\left( z, \bar{{\nu}}_z\right)= \frac{(\psi^*g)_y(z, \nu_z)}{\psi^*F(y,\nu_z)}\geq \frac{(\psi^*g)_y(z, \nu_z)}{1+\textsf{C}\rho}.
\end{align*}	
By \eqref{eqn: pull back 1} and \eqref{eqn: equal}, we have $(\psi^*g)_y(z, \nu_z) \geq (1-\rho)g_{x_0}(z, \nu_z) = (1-\rho) \psi^*F_*(y, z)$, and thus
\begin{align*}
		\frac{\psi^*F_*(y, z)}{\psi^*F_*(0,z)}&\geq \frac{1-\rho}{1+\textsf{C}  \rho} \geq 1 -2(1+\textsf{C})\rho  ,
\end{align*}
where the final inequality holds for $\rho$ small enough depending on $\textsf{C}$ and thus on $F. g$.
The same argument holds with the roles of $0$ and $y$ swapped. Together these inequalities along with \eqref{eqn: compare g, Fstar} show that 
\[
| \psi^*F_*(0,z)-\psi^*F_*(y,z)| \leq 4(1+ \textsf{C}) \, \psi^*F_*(0,z)\rho  \leq \frac{4(1+\textsf{C})}{\m} \, \rho  \,.
\]
This proves the lemma.
\end{proof}
The next simple lemma will allow us to pull back (almost) $F_*$-geodesics via the exponential map.
\begin{lemma}\label{lem: appendix geodesic in ball}
 Fix $x_1, x_2 \in M$ and let $\hat{\gamma}: [0,1]\to M$ be a curve with $\hat{\gamma}(0)=x_1$ and $\hat{\gamma}(1)=x_2$ such that $\int_0^1 F_*(\hat{\gamma}(t) , \dot{\hat{\gamma}}(t) ) \,dt \leq 2 d_{F_*}(x_1, x_2)$. Then $d_g(x_1, \hat{\gamma}(t))\leq \frac{2\M}{\m} d_g(x_1, x_2)$ for all $t \in [0,1]$.
\end{lemma}
\begin{proof} Let $\hat{\gamma}:[0,1]\to M$  be a curve as in the statement of the lemma..
 Recalling that  $d_{F_*}(x_1, x_2) \leq \m^{-1} d_g(x_1,x_2)$ by \eqref{eqn: compare g, Fstar}, we see that for any $t \in [0,1]$, 
 $$d_g(x_1, \hat{\gamma}(t)) \leq \int_0^t |\dot{\hat{\gamma}}(t)|_g \,dt \leq \M \int_0^t F_*(\hat{\gamma}(t),\dot{\hat{\gamma}}(t) )\,dt 
\leq \M \int_0^1 F_*(\hat{\gamma}(t),\dot{\hat{\gamma}}(t) )\,dt \leq \frac{2\M}{\m} d_g(x_1, x_2).
 $$
\end{proof}

Lemmas~\ref{lem: app Fstar reg} and \ref{lem: appendix geodesic in ball} will be used to prove the following lemma.
\begin{lemma}\label{last lemma}
	There exist $\rho_1= \rho_1 (g, F)>0$ and $C= C(g, F)>0$ such that for all $x_0 \in M$, $\rho<\rho_1$, and $z_1, z_2 \in B_{g_{x_0}}(0, \rho) \subset  T_{x_0} M$, we have 
	\begin{align} \label{eqn: last lemma a}
		(1-C\rho ) F_*(x_0,z_2-z_1) \leq d_{F_*}\big(\exp_{x_0}(z_1) , \exp_{x_0} (z_2)\big) \leq (1+C\rho) F_*(x_0,z_2-z_1).
	\end{align}
	In particular, up to further decreasing $\rho_1$ depending on the same parameters,
	\begin{equation}
	\label{eqn: improved 1}\begin{split}
	\Big|d_{F_*}\big(\exp_{x_0}(z_1), &\exp_{x_0} (z_2)\big)-F_*\big(x_0,z_2-z_1\big)\Big| \\
	&\leq C \rho \min\Big\{  F_*\Big(x_0,z_2-z_1\Big),  d_{F_*}\Big(\exp_{x_0}(z_1), \exp_{x_0} (z_2)\Big)	\Big\}
	\end{split}
		\end{equation}
	\end{lemma}
	\begin{proof}
	 Let $\rho_2= \rho_2(g, F)$ be chosen according to Lemma~\ref{lem: app Fstar reg}.
	  Let $\rho_1 $ be a fixed constant to be specified later in the proof, small  enough such that $\rho_1\leq (1+\frac{\m}{4\M })\rho_2$. Let $\rho <\rho_1$ and fix $z_1, z_2 \in B_{g_{x_0}}(0,\rho) \subset T_xM$. We prove the second inequality in \eqref{eqn: last lemma a} first. With the usual shorthand $\psi = \exp_{x_0}$, we have 
		\begin{align*}
			d_{F_*}(\psi(z_1),\psi(z_2) ) \leq \inf \left\{ \int_0^1 \psi^*F_*\left(\gamma(t) , \dot\gamma(t)\right)\,dx\  :\ \gamma:[0,1] \to B_{g_{x_0}}(0,\rho), \ \gamma(0)=z_1, \ \gamma(1)=z_2 \right\}.
		\end{align*}
		We plug in $\gamma(t) = tz_2+ (1-t)z_1$ as a test curve. By convexity, $\gamma(t) \in B_{g_{x_0}}(0, \rho)$ for all $t \in [0,1]$, and since $\rho<\rho_2$, we can apply  Lemma~\ref{lem: app Fstar reg} to find
		\begin{align*}
			d_{F_*}\big(\psi(&z_1), \psi(z_2)\big)
			 \leq \int_0^1 \psi^*F_*\big(tz_2 + (1-t)z_1 ,\, z_2-z_1\big) \,dt\\
			 & \leq (1+ C\rho)\int_0^1  \psi^*F_*\big(0,\, z_2-z_1\big) \,dt  = (1+ C\rho)\, \psi^*F_*\big(0,z_2-z_1\big) = (1+ C \rho)\,F_*\big(x_0,z_2-z_1\big)\,.
		\end{align*}
		Here $C = C(\textsf{C}, \m)= C(g,F)$ is the constant from \eqref{eqn: reg dual norm}.		%

		Now we prove the first inequality in \eqref{eqn: last lemma a}; the proof is similar but slightly more involved. 
		Let $\hat{\gamma}:[0,1]\to M$ be a curve with $\hat{\gamma}(0)=\psi(z_1)$ and $\hat{\gamma}(1) = \psi(z_2)$ such that 
		\[
		\int_0^1 F_*\big(\hat{\gamma}(t), \dot{\hat{\gamma}}(t)\big)\,dt \leq (1+ \rho) d_{F_*}\big(\psi(z_1), \psi(z_2)\big).
		\]
		Since, by the triangle inequality, we have $d_g(\psi(z_1), \psi(z_2))\leq 2\rho$,  Lemma~\ref{lem: appendix geodesic in ball} guarantees that the image of $\hat{\gamma}$ is contained in $B_g(\psi(z_1),\frac{4\M}{\m}\rho)$, which in turn is contained in $B_{g}(x_0,(1 + \frac{4\M}{\m})\rho)$. So, $(1 + \frac{4\M}{\m})\rho<\rho_2 < \text{inj}_{g}M/2$ we may consider the pulled-back curve $\gamma = \psi^{-1} \hat{\gamma}: [0,1] \to B_{g_{x_0}}(0,(1 + \frac{4\M}{\m})\rho) \subset T_{x_0}M$, which has $\gamma(0)= z_1,$ and $ \gamma(1) = z_2$.  It is easy to check  using duality
		  %
		  %
	that for any $z_1,z_2 \in T_{x_0}M$, 
		\begin{equation}
			\label{eqn: app f star inf curve}
		F_*(x_0,z_2-z_1) = \inf \Big\{ \psi^*F_*(0, \dot{\gamma}(t))\, dt  \ :  \ \gamma:[0,1]\to T_xM, \ \gamma(0)=z_1,\ \gamma(1)=z_2 \Big\}.
		\end{equation}
Using $\gamma$ as a competitor in \eqref{eqn: app f star inf curve} and applying the bound \eqref{eqn: reg dual norm} and  $\dot{\gamma}(t) = d\psi_{\hat\gamma(t)}(\dot{\hat{\gamma}}(t))$, we find
		\begin{align*}
			F_*(x_0,z_2-z_1) \leq \int_0^1 \psi^*F_*(0, \dot{\gamma}(t)) \,dt & \leq (1+C\rho)\int_0^1  \psi^* F_*(\gamma(t), \dot\gamma(t)) \,dt \\
 & =(1+C\rho)\int_0^1 F_*(\hat{\gamma}(t), \dot{\hat{\gamma}}(t)) \,dt   = (1+C\rho) d_{F_*}(\psi(z_1), \psi(z_2)).
		\end{align*}
		Here $C = C(\textsf{C}, \m ,\M)= C(F,g)$. This proves \eqref{eqn: last lemma a}; \eqref{eqn: improved 1} follows immediately from \eqref{eqn: last lemma a} up to decreasing $\rho_1$ depending on $C $ (and thus on $ F,g$)  and doubling the constant $C$.
	\end{proof}
We are now ready to prove Lemma~\ref{lem: appendix hausdorff}.
\begin{proof}[Proof of Lemma~\ref{lem: appendix hausdorff}]
	 Let $\rho_1=\rho_1(g, F)$ be chosen according to Lemma~\ref{last lemma}.
Let $\rho_0 >0$ be a fixed constant to be specified in the proof, small enough so that 
$(1+ \m^{-1})\rho_0\leq  \rho_1$. 
Fix $r \in (0,\rho_0]$ and $\rho \in [0,\rho_0]$ and choose $x_0, x_1 \in M$ with $d_g(x_0, x_1)\leq r$. Let $\psi = \exp_{x_0} :T_{x_0}M \to M$. Provided we choose $\rho_0$ small enough in terms of $g, \m , \M$, we can pull back ${B}_{F_*}(x_1,2\rho_0)$ by $\psi$ and it suffices to show that 
	\begin{align}
\label{1}		
d_{H, g_{x_0}}\big(\partial r K_{x_0} +z_1,\, \psi^{-1} (\partial B_{F_*}(x_1,r))\big) \leq C \rho \,r  ,\\
		\label{2} d_{H, g_{x_0}}\big(r K_{x_0} +z_1,\, \psi^{-1} ( B_{F_*}(x_1,r))\big) \leq C \rho \, r\,.
	\end{align}
We prove \eqref{1}, with the proof of \eqref{2} being analogous. Toward \eqref{1}, fix $z_2 \in \partial r K_{x_0} + z_1 \subset T_{x_0}M$. So, $F_*(x_0, z_2-z_1) =r$. 
By assumption $|z_1|_{g_{x_0}}<\rho$, and thus using \eqref{eqn: compare g, Fstar}, we also have $|z_2|_{g_{x_0}}<\rho + \m^{-1}r \leq \rho_1.$ So, we can apply Lemma~\ref{last lemma}: recalling that  $F_*(x_0, z_2-z_1) =r$, \eqref{eqn: improved 1} guarantees that
\begin{equation}
	\label{app eqn: claim a}
|d_{F_*}(x_1, \psi(z_2))-r |  \leq C\rho\, r.
\end{equation} 
 Note that \eqref{app eqn: claim a} implies $z \in \psi^{-1}(B_{F_*}(x_1,\hat{r}))$ for $\hat{r}$ with $|\hat{r}-r| < C\rho \, r$, Together with \eqref{eqn: compare g, Fstar} this implies
\begin{equation}
	\label{eqn: 1 contain 1}
 z_2 \in \Nb{g_{x_0}}{\psi^{-1}(\partial B_{F_*}(x_1, r))}{\,{C\rho\, r}}\,
\end{equation}
where $C= C(g, F)$.

The other direction is analogous. Take $z_2 \in \psi^{-1}(\partial B_{F_*}(x_1,r))$ so that $d_{F_*}(x_1,\psi(z_2))=r.$ Again we have assumed that $|z_1|_{g_{x_0}}<\rho$ and using \eqref{eqn: compare g, Fstar} deduce that  $|z_2|_{g_{x_0}}< \rho + \m^{-1}r < \rho_1$ as well. Hence, we are in a position to apply Lemma~\ref{last lemma}, which guarantees that  
\[
|F_*(x_0, z_2-z_1) -r | \leq C \rho \,r\,.
\]
So, we see that $z_2 \in \partial K \hat{r} +z_1$ for some $\hat{r}$ with $|\hat{r}-r| <C \rho\,r.$ Together with \eqref{eqn: compare g, Fstar}, this proves that 
\[
z_2 \in \Nb{g_{x_0}}{\partial r K_{x_0} + z_1}{\,{C \rho\,r}}\,
\] with $C= C(g, F)$ and completes the proof. 
\end{proof}

\section{The  diameter bound}\label{sec:diam}
In this section we prove a uniform scale-invariant  diameter bound
 for  volume-constrained $\eo$-local minimizers $\W_v$ of $\F$ with sufficiently small volume $v$: $\W_v$  is contained in the union of $J_0$ balls of radius $2 v^{\sfrac{1}{n}}.$ This estimate
 is uniform in $v$ in the sense that,  with respect to the rescaled metric $h=v^{\sfrac{-2}{n}}g$, $|\W_v|_h =1$ and $\W_v$ is contained in $J_0$ balls of radius $2$.
\begin{theorem}\label{thm: diameter bound}
Let $(M,g)$ be a closed Riemannian manifold of dimension $n\geq2$ and let $\F$ be an anisotropic surface energy with integrand $F. $ 
Fix $\eo>0$ and $\kappa>0$. There exist $v_0=v_0(n,g,F,\eo,\kappa) \in (0,|M|_g)$ and $\JO=\JO(n,g, F,\eo,\kappa)\in \mathbb N$ such that for any volume-constrained $\eo$-local minimizer  $\W_v$ of $\F$ with volume $v < v_0$ and $\F(\W_v) \leq \kappa v^{\sfrac{(n-1)}{n}}$, there is a collection of points $x_1, \dots ,x_{\JO} \in M$ such that  
	$$\W_v\subset \bigcup_{i=1}^{\JO} B_g\big(x_i, 2v^{\sfrac{1}{n}}\big).$$
	\end{theorem}

The proof of Theorem~\ref{thm: diameter bound} has three steps. First, Lemma~\ref{lem: J balls} uses the concavity of the function $t \mapsto t^{\sfrac{(n-1)}n}$ to imply that if a set of finite perimeter is the union of $N$ disjoint sets, then a significant portion of its volume is contained in the union of $\JO$ of those sets. Next, combining Lemma~\ref{lem: J balls} with a grid argument inspired by \cite{Figalli-Maggi-log, MaggiMihaila}, we show in Proposition \ref{prop: measure estimate} that any set of finite perimeter with volume $v$ and a uniform perimeter (or $\F)$ bound has most of its volume contained in the union of $\JO$ balls of radius $v^{\sfrac{1}{n}}$. Finally,  we prove a differential inequality that allows us to improve this measure bound to containment in $\JO$ balls in the case of a  volume-constrained $\eo$-local minimizer.

\subsection{A lemma about concavity and sequences of real numbers}\label{ssec: concavity}
The following lemma says that if a nonnegative decreasing sequence $\{a_i\}$ sums to $1$ and has $\|\{a_i\}\|_{\ell^\alpha}$ bounded for a concave power $\alpha \in (0,1)$, then the tail end of the sequence has small $\ell^1$ norm. This lemma has been already used in concentration compactness arguments, see for instance \cite[Proposition 3.7]{Gol22}, \cite[Proposition 3.1]{candautilh2021existence}, \cite[Lemma 5.6, Lemma 6.6]{CGOS}, \cite[Theorem 3.3]{novaga2022isoperimetric}. We will apply it with $\alpha = \frac{n-1}{n}$ as described above.
\begin{lemma}\label{lem: J balls}
	Fix $\alpha \in (0,1),$ $\kappa >0$, and $\ETA>0$. There exists $\JO = \JO (\alpha,\kappa , \ETA) \in \mathbb{N}$ such that for any sequence $\{a_i\}_{i\in \mathbb{N}}$ of nonnegative real numbers with $a_1 \geq a_2 \geq \dots $ and such that
	\begin{align*}
		\sum_{i\in \mathbb{N}} a_i = 1
		 \qquad \text{ and } \qquad 
		\sum_{i \in \mathbb{N}} a_i^{\alpha} \leq \kappa,
	\end{align*}
	we have $\sum_{i=1}^{\JO} a_i \geq 1-\ETA. $
\end{lemma}

\begin{proof}
Since $\sum_{i=1}^J a_i \leq 1$ and the sequence is decreasing, we observe that $a_J \leq 1/J$ for every $J\in \mathbb N$. Hence we compute
$$1- \sum_{i=1}^{J} a_i =\sum_{i> J} a_i=\sum_{i > J} a_i^{\alpha} a_i^{1-\alpha}\leq a_J^{1-\alpha}\sum_{i > J} a_i^{\alpha} \leq \frac 1{J^{1-\alpha}}\kappa.$$
The proof follows choosing $J$ large enough so that $\frac 1{J^{1-\alpha}}\kappa \leq \eta$.
\end{proof}

\subsection{A  diameter bound in measure}\label{ssec: measure bound}
In this section, we prove that a set of finite perimeter in $(M,g)$ with small enough volume $v$ has all but an $\eta$-fraction contained in $\JO$ balls of radius $v^{1/n}$, where $\JO$ depends only on dimension, the scale invariant anisotropic perimeter bound and $\eta$. The proposition applies to all sets of finite perimeter of sufficiently small volume and does not require  minimality. %
\begin{proposition}
\label{prop: measure estimate}
	Fix $\kappa \geq 1$ and $\ETA>0$. Let $(M,g)$ be a closed Riemannian manifold of dimension $n \geq 2$ and fix an anisotropic surface energy $\F$ with integrand $F.$
There exist $v_0=v_0(n, g,F,\ETA)>0$ and $\JO = \JO (n , F,\kappa , \ETA) \in \mathbb{N}$ such that the following holds. For any finite perimeter set $\W$ with volume $v \in (0,v_0]$ and $\F(\W)\leq \kappa v^{(n-1)/n}$, we may find points $x_1, \dots x_{\JO}$ in $M$ such that 
\begin{equation}\label{eqn: measure est}
	\Big|\W \setminus \bigcup_{i=1}^{\JO} B_g\big(x_i, v^{\sfrac{1}{n}}\big)\Big|_g \leq \ETA v\,.
\end{equation}
\end{proposition} 
The idea of the proof of Proposition~\ref{prop: measure estimate} is the following. First, we intersect $\W$ with a collection of disjoint open sets $\{Q_i\}$ of diameter at most $v^{\sfrac{1}{n}}$ that cover $\W$ up to a negligible set,  decomposing $\W$ into finitely many pairwise disjoint sets with the desired diameter bound. Importantly,  Lemma~\ref{cubecovering} below ensures  the collection $\{Q_i\}$ can be constructed in such a way that we quantitatively control the amount of surface energy that is added through taking intersections. Next, the Wulff inequality \eqref{eqn: Wulff small volume} yields a bound on the sum of a {\it concave power} of the volume of each component. Finally,  by Lemma~\ref{lem: J balls} we conclude that most of the volume of $\W$ must be contained in $\JO$ of the disjoint components.
 \begin{remark}
 	{\rm 
 	The smallness of $v_0$ in the statement of Proposition~\ref{prop: measure estimate} is used to apply the Wulff inequality in the form \eqref{eqn: Wulff small volume} on $(M,g)$. In Euclidean space with a translation invariant $\F$ (in particular the perimeter), the Wulff inequality \eqref{eqn: Wulff small volume} holds for every volume: hence Proposition~\ref{prop: measure estimate} holds, with the same proof, for sets of finite perimeter of any volume.
 	}
 \end{remark}
 
\begin{lemma}\label{cubecovering}
Let $(M,g)$ be a closed Riemannian manifold of dimension $n \geq 2$. There exist $c=c(n,g)>0$ and $r_0=r_0(n,g)>0$ such that the following holds. For every finite perimeter set $E\subset M$ and  $r\in (0,r_0)$, there is a finite collection of pairwise disjoint open sets $\{Q_i\}_{i=1}^N$ in $M$ with $\text{diam}_g(Q_i)\leq r$ such that $|E \setminus \cup_{i=1}^N Q_i |_g=0$, 
\begin{equation}\label{boundcover0}
\mathcal H^{n-1}_g(\{x\in \partial^*E\cap \partial^*Q_i \, : \, \nu_{E}(x)=\nu_{Q_i}(x)\})=0,
\end{equation} and 
\begin{equation}\label{boundcover}
\frac {|E|_g}r \geq c\, \sum_{i=1}^N \mathcal H^{n-1}_g \big(E^{(1)}\cap \partial Q_i\big).
\end{equation}
\end{lemma}
In \cite[Lemma 5.1]{Figalli-Maggi-log}, a statement analogous to Lemma~\ref{cubecovering} is shown in Euclidean space by decomposing $\R^n$ into cubes whose sides are parallel to a judiciously chosen orthonormal basis depending on the set $E$ itself. To prove Lemma~\ref{cubecovering}, we will obtain an initial collection of sets by applying \cite[Lemma 5.1]{Figalli-Maggi-log} in charts, and then refine the sets by hand to ensure they satisfy the properties of the lemma. Although \cite[Lemma 5.1]{Figalli-Maggi-log} is stated without using the set $E^{(1)}$ of points of density $1$ and without the property \eqref{boundcover0}, as observed in \cite[Proof of Lemma 3.1, Step four]{MaggiMihaila} we will need to use $E^{(1)}$ and \eqref{boundcover0} in the proof of Proposition~\ref{prop: measure estimate} in order to apply \cite[Theorem 16.3 (16.7)]{Mag} to obtain \eqref{usee}.
\begin{proof}[Proof of Lemma~\ref{cubecovering}]
{\it Step 1.} First, we construct an initial collection of open sets $\{Q_i\}_{i=1}^N$, possibly not pairwise disjoint,  with diameter at most $r$ that cover $E$ up to a negligible set and satisfy \eqref{boundcover}.

Choose $r_0 < (\text{inj}_g M)/4$ small enough so that $\frac{1}{2} g_{euc} \leq \psi^*g \leq 2 g_{euc}$ on $B_{g_{euc}}(0, 4r_0) \subset \R^n,$ where $\psi$ is the normal coordinate map centered at any $x\in M$.
 Choose a finite collection of points $x_1 ,
\dots , x_K$ such that the balls $\{B_g(x_k, 2r_0)\}_{k=1}^K$ cover $M$ and let $\psi_k : B_{g_{euc}}(0, 4r_0) \to M$ be the normal coordinate map centered at $x_k$.

Let $G_k = \psi_k^{-1}( E \cap B_g(x_k, 3r_0)) \subset  \R^n$, so that $G_k \Subset B_{g_{euc}}(0, 4r_0)$ and $E = \cup_{k=1}^K \psi_k(G_k)$.   Fix $k \in \{1,\dots, K\}$ and $r<r_0.$  
 Applying \cite[Lemma 5.1]{Figalli-Maggi-log}  to the set $G_k $, we obtain a collection of disjoint open cubes $\{Q'_k\}$ of diameter $r/2$ with parallel sides that cover Lebesgue almost all of $\R^n$ such that 
$	{|G_k|_{g_{euc}}} \geq \frac{r}{4n} \sum_{Q_k'} \mathcal{H}^{n-1}_{g_{euc}} (G_k^{(1)} \cap \pa Q_{k}').
$
Let $\{Q'_{k,a}\}_{a \in A'_k} \subset \{ Q'_k\}$ be the finite collection of those cubes that intersect $G_k$ nontrivially. Note that $Q'_{k,a} \subset B_{g_{euc}}(0, 4r_0)$ for each $a \in A_k'$, and 
\begin{equation}\label{figallimaggi}
	{|G_k|_{g_{euc}}} \geq \frac{r}{4n} \sum_{a \in A_k'} \mathcal{H}^{n-1}_{g_{euc}} \big(G_k^{(1)} \cap \pa Q_{k,a}'\big)\,.
\end{equation}

For each $a \in A_k'$, let $Q_{k,a}= \psi_k(Q'_{k,a})$. Notice that $\text{diam}_g(Q_{k,a}) \leq r$ for all $k\in\{1,\dots, K\}$ and $a \in A_k'$. As an initial refinement of this collection, we let $A_1= A_1'$ and for $k \geq 2$ let
$$A_k := \Big\{ a \in A_k'\ :\ Q_{k,a} \not \subset \bigcup_{j<k} \bigcup_{a \in A_j} Q_{j,a}\Big\}.$$
The collection $\{Q_{k,a}\}_{a \in A_k, 1\leq k \leq K}$ consists of open sets with $\text{diam}_g(Q_{k,a})\leq r$ and covers $E$ up to a set of measure zero. Moreover, applying \eqref{figallimaggi} in charts, we find that
\begin{equation*}\label{eqn: middle est}
	\begin{split}
{|E|_g} \geq \frac{1}{K} \sum_{k=1}^K |E\cap V_k|_g
 \geq \frac{1}{2K} \sum_{k=1}^K |G_k|_{g_{euc}}
&  \geq \frac{r}{8nK} \sum_{k=1}^K \sum_{a \in A_k} \mathcal{H}^{n-1}_{g_{euc}}\big(G_k^{(1)}\cap \partial Q_{a,k}'\big)\\
&  \geq \frac{r}{16nK} \sum_{k=1}^K \sum_{a \in A_k} \mathcal{H}_{g}^{n-1}\big(E^{(1)}\cap \partial Q_{a,k}\big),
\end{split}
\end{equation*}
 so the collection satisfies \eqref{boundcover}. However, the sets in this collection are not pairwise disjoint.

{\it Step 2:}   We now slightly modify the collection of sets from Step 1 above so that they are pairwise disjoint and are still open with diameter at most $r$, cover $E$ up to a negligible set, and satisfy the estimate \eqref{boundcover}. Fix $2 \leq k \leq K$ and $b \in A_k$. Let 
$I_{k,b}= \{ a \in A_1 : Q_{1,a} \cap Q_{k,b} \neq \emptyset \}$
be the indices corresponding to cubes from the chart $\psi_1$ that intersect $Q_{k,b}.$
By the construction from disjoint cubes in charts, the cardinality of $I_{k,b}$ is at most $C_n$.
Let 
$
\hat{Q}_{k,b} := Q_{k,b} \setminus \medcup_{a \in I_{k,b}}   \overline{Q}_{1,a}.
$
Then 
\begin{align*}
\mathcal{H}_g^{n-1}\big(\partial \hat{Q}_{k,b} \cap E^{(1)}\big) 
& \leq \mathcal{H}_g^{n-1} \big(\partial Q_{k,b} \cap E^{(1)}\big) + \sum_{a \in I_{k,b}} \mathcal{H}_g^{n-1} \big(\partial Q_{1,a} \cap {Q}_{k,b} \cap E^{(1)}\big)\,.
\end{align*}
Summing this up over all $ b \in A_k$ and $2 \leq k \leq K$, we 
find that 
\begin{align*}
	\sum_{k=2}^{K}\sum_{b \in A_k} \mathcal{H}^{n-1}_g\big(\partial \hat{Q}_{k,b} \cap E^{(1)}\big) \leq \sum_{k=2}^{K}\sum_{b \in A_k}  \mathcal{H}^{n-1}_g \big(\partial Q_{k,b} \cap E^{(1)}\big)
	 + C_nK \sum_{a \in A_1} \mathcal{H}^{n-1}_g \big( \partial Q_{1,a} \cap E^{(1)} \big)\,.
\end{align*}
Here we have used the fact that any $x \in \partial Q_{1,a}$ is contained in ${Q}_{k,b}$ for at most $C_n K$ cubes, thanks to the construction from disjoint cubes in charts.
Adding $\sum_{a \in A_1} \mathcal{H}_g^{n-1} \big( \partial Q_{1,a} \cap E^{(1)} \big)$ to both sides and recalling \eqref{figallimaggi}, we see that
 \[
\sum_{a \in A_1} \mathcal{H}_g^{n-1} \big( \partial Q_{1,a} \cap E^{(1)} \big) + \sum_{k=2}^{K}\sum_{b \in A_k} \mathcal{H}_g^{n-1}\big(\partial \hat{Q}_{k,b} \cap E^{(1)}\big) 
\leq C_n K \sum_{k=1}^K \sum_{a \in A_k} \mathcal{H}_{g}^{n-1}(E^{(1)}\cap \partial Q_{a,k})\leq \frac{|E|_g}{r}.
 \]
So, the collection of sets $\{ Q_{1,a}\}_{a \in A_1}  \cup \{ \hat{Q}_{k,b} \}_{2 \leq k\leq K, b \in A_k}$ satisfies \eqref{boundcover}, each set is open with diameter at most $r$, and the sets $Q_{1,a}$ are pairwise disjoint and also have trivial intersection with any $\hat{Q}_{k,b}$. 

Setting aside the sets $\{Q_{1,a}\}_{a \in A_1}$, we apply the same  procedure 
with the index $k=2$ playing the role of $1$ to refine the sets $\{\hat{Q}_{k,b}\}$ for $3\leq k \leq K$, $b \in A_k$, to make them disjoint from $\hat{Q}_{2,a}$ for any ${a \in A_2}$ and satisfy the properties above. Proceeding inductively and applying the refinement procedure $K$ times, we obtain a collection of sets satisfying the properties of the lemma. In particular, property \eqref{boundcover0} can be obtained by slightly tilting the initial collection of open sets $\{Q_i\}_{i=1}^N$.
\end{proof}

We now prove Proposition~\ref{prop: measure estimate}.

\begin{proof}[Proof of Proposition~\ref{prop: measure estimate}] Let $\bar{v}$ be as in \eqref{eqn: Wulff small volume}, let $r_0$ be as in Lemma \ref{cubecovering}, and set $v_0:=\min\{\bar v, r_0^n\}$. Let $\{Q_i\}_{i=1}^N$ be the collection of sets obtained applying Lemma \ref{cubecovering} to $E= \W$ with $r=v^{\sfrac{1}{n}}$. 
We first apply the isoperimetric inequality \eqref{eqn: Wulff small volume} and then, using \eqref{boundcover0}, we apply \cite[Theorem 16.3 (16.7)]{Mag} to compute
\begin{equation}\label{usee}
\begin{split}
\sum_{i=1}^N|\W\cap Q_i|_g^{\frac{n-1}n}\leq C \sum_{i=1}^N\F(\W \cap Q_i) \leq C\Big( \F(\W) +\sum_{i=1}^N \mathcal H^{n-1}_g (\W^{(1)}\cap \partial Q_i)\Big).
\end{split}
\end{equation}
Applying estimate \eqref{boundcover} from Lemma~\ref{cubecovering} to \eqref{usee}, we obtain
\begin{equation*}
\begin{split}
\sum_{i=1}^N|\W\cap Q_i|_g^{\frac{n-1}n}\leq  C \left(\kappa v^{\frac{n-1}n} + \frac{v}{r}\right).
\end{split}
\end{equation*}
Dividing by $v^{\frac{n-1}n}$ and using the choice $r=v^{\frac{1}{n}}$, we deduce that
\begin{equation}
 \sum_{i=1}^N\left( \frac{|\W\cap Q_i|_g}{v}\right)^{\frac{n-1}n}  \leq C \frac{v^{\sfrac{1}{n}}}{r} + C\kappa \leq C\kappa.
\end{equation}
The sets $Q_i$ are pairwise disjoint and cover $E$ up to a set of measure zero, so $\sum_{i=1}^N \frac{|\W \cap Q_i|_g}{v} = 1$.  Up to relabeling the indices, we can suppose that the sequence $a_i:=\frac{|\W\cap Q_i|_g}{v}$ is non-increasing, hence we can apply Lemma \ref{lem: J balls}  to $\{a_i\}_{i=1}^N$, to deduce that there exists $\JO = \JO (n ,\kappa , \ETA) \in \mathbb{N}$ such that 
$$\sum_{i=1}^{\JO} \frac{|\W\cap Q_i|_g}{v}   \geq 1-\ETA.$$
For each $i = 1, \dots , \JO,$ since $\text{diam}_g(Q_i)\leq r= v^{\frac{1}{n}}$, we can find a point $x_i \in M$  such that $Q_i \subset B_g(x_i, v^{1/n})$. So, again using the pairwise disjointness of the $Q_i$, 
\begin{equation}
\begin{split}
\big|\W\setminus \mlcup_{i=1}^{\JO} B_g(x_i, v^{1/n})\big|_g &\leq \big|\W_v \setminus \mlcup_{i=1}^{\JO} Q_i\big|_g = |\W|_g - \sum_{i=1}^{\JO} |\W\cap Q_i|_g \leq  \ETA v.
\end{split}
\end{equation}
Thus \eqref{eqn: measure est} holds, as desired.
\end{proof}

\subsection{Proof of the  diameter bound}\label{ssec: diameter bound}
In this section, we prove the  diameter bound of Theorem~\ref{thm: diameter bound}. In the proof, we will use (slight modifications of) sets of the form $\W_v':=\W_v\cap \bigcup_{i=1}^{\JO} B_g(x_i,R)$ as  competitors for the $\eo$-local minimality of $\W_v$, where $x_i,\dots , x_\JO$ are the points obtained in Proposition~\ref{prop: measure estimate}. To this aim, we first prove that $\W_v \Delta \W_v' \subset \Nb{g}{\partial \W_v}{\delta}$ in the following lemma. Recall that $\Nb{g}{E}{\delta}$ is the tubular neighborhood defined in \eqref{eqn: tube def}.
\begin{lemma}
 \label{lem: containment}
	Fix a Riemannian manifold $(M,g)$ of dimension $n\geq 2$. 
	Let $r_0>0$ be small enough so that $|B_g(x,r)|_g\geq \omega_n r^n/2 $ for any $x \in M$ and $r \in (0,r_0).$ Fix $\delta>0$ and $\JO\in \mathbb{N}$. Let $\gamma < \frac{\omega_n}{2} \min\{ r_0^n, \delta^n\}.$ If $\W\subset M$ is a measurable set with 
\[
\left|\W\setminus \mlcup_{i=1}^{\JO} B_g(x_i,R)\right|_g \leq\gamma
\]
for some $x_1,\dots, x_\JO\in M$ and $R>0$, then 
\begin{equation}
	\label{eqn: contain1}
\W \setminus \mlcup_{i=1}^{\JO}B_g(x_i, R+\delta) \subset \Nb{g}{\partial \W}{\delta}.
\end{equation}
\end{lemma}
\begin{proof}
	Suppose there is a point $x_0 \in \W \setminus  \bigcup_{i=1}^{\JO}B_g(x_i, R+\delta)$ with $x_0 \not\in  \Nb{g}{\partial \W}{\delta}$. Then by definition, 
	$$
	B_g(x_0, \delta ) \subset \W \setminus  \mlcup_{i=1}^{\JO}B_g(x_i, R)
	$$ and so in particular
	\[
	|B_g(x_0,\delta)|_g \leq \left| \W \setminus \mlcup_{i=1}^{\JO}B_g(x_i, R)\right|_g \leq \gamma.
	\]
	On the other hand, $|B_g(x_0,\delta)|_g \geq \frac{\omega_n}{2} \min\{ r_0^n, \delta^n\}$, contradicting our choice of $\gamma.$ We conclude that no such point exists and the containment \eqref{eqn: contain1} holds.
\end{proof}
We are now ready to prove the main result of the section, using Lemma~\ref{lem: containment} and Proposition~\ref{prop: measure estimate} to establish a differential inequality for the volume of $\W_v$ outside $\JO$ balls of radius $r$.

\begin{proof}[Proof of Theorem~\ref{thm: diameter bound}] We begin by fixing parameters.
Let $\ETA =\ETA(n, \eo) \in(0,1/2)$ be a fixed number to be determined later in the proof. Let $\JO = \JO(n, F, \kappa, \ETA) = \JO(n, F,\kappa, \eo)$ be chosen according to Proposition~\ref{prop: measure estimate}. Choose $r_0= r_0(g, \ETA, \JO,F) = r_0(g, n,\kappa,F) <\text{inj}_g M$  to be small enough according to the assumptions of Lemma~\ref{lem: containment} and such that
\begin{equation}
	\label{eqn: metrics close}
	\| \psi^*g - g_{euc} \|_{C^1(B_{g_{euc}}(0,r_0))}  \leq \frac\ETA {12 \JO} \qquad \text{ and } \qquad  \sup_{\nu \in S^{n-1}}\| \psi^*F(\cdot, \nu)\|_{C^1(B_{g_{euc}}(0,r_0))} \leq 1,
\end{equation}
in normal coordinates $\psi$ at any $x\in M$.  Let $v_0= v_0(n,g, F,\kappa, \ETA)< r_0^n \om_n /2$ be small enough to apply Proposition~\ref{prop: measure estimate} and such that $3v^{\sfrac{1}{n}}_0 < r_0 / 100 \JO$ . Let $v<v_0$ be fixed. 
Throughout the proof, $c_n$ denotes a dimensional constant whose value may change from line to line.

By Proposition~\ref{prop: measure estimate}, we can find a collection of points $x_1, \dots ,x_{\JO} \in M$ such that 
\begin{equation}\label{eqn: meas est diam bound}
	\big| \W_v \setminus \mlcup_{i=1}^{\JO} B_g(x_i, v^{\sfrac{1}{n}})\big|_g < \ETA v\,.
\end{equation}
Let  $I = [v^{\sfrac{1}{n}}, 3v^{\sfrac{1}{n}}]$, and for $r \in I$ let 
\[
A(r) = \bigcup_{i=1}^{\JO} B_g(x_i, r) \qquad \text{ and } \qquad
u(r) = \frac{|\W_v \setminus A(r)|_g}{v}.
\]
 Note that $u$ is decreasing in  $r$, and $u(v^{\sfrac{1}{n}}) < \ETA$ by \eqref{eqn: meas est diam bound}. We claim there exists $c_n>0$ such that 
\begin{equation}
	\label{eqn: diff ineq}
[(v u(r) )^{\sfrac{1}{n}}]' \leq -c_n
\end{equation}
for all $r \in I$ with $u(r)>0$. Before proving the differential inequality \eqref{eqn: diff ineq}, let us see how it will allow us to conclude the proof of the proposition. Take $r \in I$ such that $u(r)>0$. Since $u$ is decreasing,  \eqref{eqn: diff ineq} holds for all $s \in (v^{\sfrac{1}{n}}, r].$ Integrating this differential inequality and recalling that $u(v^{\sfrac{1}{n}})<\ETA$, we find
\[
 c_n \big(\hat{r} - v^{\sfrac{1}{n}}\big) \leq (\ETA v)^{\sfrac{1}{n}}- [v u(\hat{r})]^{\sfrac{1}{n}} <(\ETA v)^{\sfrac{1}{n}} .
\]
In particular, provided we choose $\ETA < c_n/2$, we find that  $r < 2 v^{\sfrac{1}{n}}$. Thus $u$ vanishes on $[2v^{\sfrac{1}{n}}, 3v^{\sfrac{1}{n}}]$ and $\Omega_v \subset A(2v^{\sfrac{1}{n}})$. This shows the claim in the proposition.
\\

It  therefore remains to prove \eqref{eqn: diff ineq}. By the coarea formula, we find that
\begin{equation}
	\label{eqn: u prime}
u'(r) = -\frac{1}{v}\, \mathcal{H}_g^{n-1}(\partial A(r) \cap \W_v)
\end{equation}
for a.e. $r \in I$. To gain information about the right-hand side of \eqref{eqn: u prime}, we
 would like to use the sets 
 $$E_r := A(r) \cap \W_v$$
  for $r \in I$ as competitors for the local minimality of $\W_v$. But, since we may have $|E_r| <v$, we must modify the sets using Lemma~\ref{lemma: dilation} to make them admissible competitors. To this end, note that $A(3v^{\sfrac{1}{n}})$ has $1\leq K \leq \JO$ connected components $A_1, \dots , A_K$, and each connected component $A_k$ has diameter at most $6\JO\, v^{\sfrac{1}{n}}$. 
Thus, we can find a collection of disjoint open sets $V_1,\dots, V_K$ in $M$ such that $\text{diam}_g(V_k)\leq 12\JO\,v^{\sfrac{1}{n}}$ and $\Nb{g}{A_k}{v^{\sfrac{1}{n}} }\subset V_k.$ In particular, $\Nb{g}{E_r}{v^{\sfrac{1}{n}}}\subset \medcup_{k=1}^K V_k$ for each $r \in I$.  

In terms of the rescaled metric $h= v^{\sfrac{-2}{n}} g$, this means that $\text{diam}_h (V_k) \leq 12\JO$, and  $\Nb{h}{E_r}{1} \subset \medcup_{k=1}^K V_k$, and $|E_r|_h = 1 - u(r) \in [1-\ETA, 1]$. Moreover, since $v_0$ was chosen small enough  that $v^{\sfrac{1}{n}} <r_0/12\JO$, the estimates \eqref{eqn: metrics close} hold with $h$ in place of $g$ and $B(0, 12\JO)$ in place of $B(0,r_0)$. Thus, for each $r\in I$, we may apply Lemma~\ref{lemma: dilation} on $(M, h)$  with $D= 12 \JO$ and $E =E_r$. In terms of the metric $g$, the conclusion of the lemma tells us there is a set $\Emod_r$ with $|\Emod_r| =v$ such that  
\begin{align}
\label{eqn: emod rescale}
	\Emod_r \Delta E_r & \Subset \Nb{g}{\partial E_r}{c_n \JO \ETA v^{\sfrac{1}{n}}} ,\\
 	\label{eqn: energy compare}
 	\F(\Emod_{r}) & \leq \left\{ 1 +c_n \JO u(r)\right\} \F(E_r) .
\end{align}

We now claim that, if $\ETA$ is chosen to be small enough, we have 
\begin{equation}\label{symmdiff}
 \Emod_{r}\Delta \Omega_v \subset \Nb{g}{\partial \Omega_v}{\eo v^{\sfrac{1}{n}}},
\end{equation}
thus $\Emod_{r}$ is an admissible competitor for the local minimality of $\Omega_v$. 
Thanks to the triangle inequality property of the symmetric difference \eqref{eqn: triangle}, 
it suffices to show that $\W_v \Delta E_r$ and $E_r \Delta \Emod_{r}$ are both contained in this neighborhood of $\partial \W_v$. 
To obtain the first containment, we apply Lemma~\ref{lem: containment} with $\delta = \frac{\eo}{2} v^{\sfrac{1}{n}}$ and $\gamma = \ETA v$, with $\ETA>0$ chosen small enough depending on $\eo$ so that $\ETA v < \frac{\omega_n}{2} \min\{r_0^n, \e_0^n v/2^n\}$. This implies 
\begin{equation}
	\label{eqn: contain 2}
E_r \Delta \W_v = \W_v \setminus A(r) \subset \Nb{g}{\partial \W_v}{{\frac{\eo}{2} v^{\sfrac{1}{n}}}}.
\end{equation}
Next, to show the containment of $E_r \Delta \Emod_{r}$, thanks to \eqref{eqn: emod rescale}, it suffices to show that 
$$\Nb{g}{\pa E_r}{c_n \JO \ETA v^{\sfrac{1}{n}}} \subset \Nb{g}{\partial \W_v}{\eo v^{\sfrac{1}{n}}}.
$$
 Since $\partial E_r = \left( \partial \W_v \cap A(r)\right) \cup \left( \partial A(r)\cap \W_v \right) $ we thus find that
\begin{align*}
	\Nb{g}{\partial E_r}{c_n \JO \ETA r} &= \Nb{g}{\partial \W_v \cap A(r)}{c_n\JO \ETA r}   \ \medcup \ \Nb{g}{\partial A(r)\cap \W_v}{c_n\JO \ETA r} \\
	& \subset \Nb{g}{\partial \W_v}{c_n\JO \ETA r}  \ \medcup \ \Nb{g}{\partial A(r)\cap \W_v}{c_n \JO\ETA r} \,.
\end{align*}  
 Since $\partial A(r)\cap \W_v \subset \W_v \setminus A(r)$ and  $r\leq 3v^{\sfrac{1}{n}}$, if we take $c_n \JO \ETA r\leq \eo v^{\sfrac{1}{n}}/2$, we have $\Nb{g}{\partial \W_v \cap A(r)}{c_n \JO\ETA r}  \subset \Nb{g}{\partial \Omega_v}{\eo v^{\sfrac{1}{n}}}$ by \eqref{eqn: contain 2} above. Thus \eqref{symmdiff} holds. 
 
 Hence, $\Emod_r$ is an admissible competitor for the local minimality of $\W_v$, and  so
   we find that 
\begin{equation*}
\begin{split}
\F(\W_v)&\leq \F(\Emod_{r})\overset{\eqref{eqn: energy compare}}{\leq} (1 + C u(r))\F(E_r)\\
&=
(1+Cu(r)) \left(\F(\W_v ; A(r)) + C\mathcal{H}^{n-1}(\partial A(r) \cap \W_v)\right)\\
&\overset{\eqref{eqn: u prime}}{\leq}(1+C u(r))(\F(\W_v ; A(r)) +Cv|u'(r)|) \\
& \leq \F(\W_v ; A(r)) +C \kappa v (v^{-\sfrac 1n}u(r)+ |u'(r)|).
\end{split}
\end{equation*}
Subtracting $\F(\W_v; A(r))$ from both sides and adding $\int_{\partial A(r)\cap \W_v} F(x, -\nu_{A(r)}) \,d\mathcal{H}^{n-1}$ to both sides (and noting the latter term is bounded above by $ C v |u'(r)|$), we find
\begin{align*}
	\F(\W_v \setminus A(r)) \leq C \kappa v (v^{-\sfrac 1n}u(r)+ |u'(r)|)
\end{align*}
from which we deduce from \eqref{eqn: Wulff small volume} that
$$(vu(r))^{\sfrac{(n-1)}{n}} \leq C\kappa v(v^{-\sfrac 1n}u(r)+ |u'(r)|).$$
Choosing $\ETA$ small enough that  $C \kappa v u(r)\leq \frac{1}{2} u(r)^{\sfrac{(n-1)}{n}}$, we obtain  \eqref{eqn: diff ineq} and conclude the proof.
\end{proof}

\section{Uniform convergence to a Wulff shape, qualitatively}\label{sec: qual}  

In this section, we prove Theorem~\ref{thm: uniform convergence2}, showing that  for $v$ sufficiently small, a volume-constrained $\eo$-local minimizer $\W_v$  of $\F$ is uniformly close to a tangent Wulff shape of the appropriate volume at some point $x_0 \in M$. At this stage, the estimates are qualitative with respect to the volume parameter $v$.

\begin{theorem}\label{thm: uniform convergence2}
	Fix a closed Riemannian $n$-manifold $(M,g)$ and an anisotropic surface energy $\F$ with integrand $F$. For every $\kappa>0$ and $\eo>0$, there exists $v_0  = v_0(n,g, F, \kappa, \eo) \in (0, |M|_g)$  such that the following holds. Let $\W_v$ be a volume-constrained $\eo$-local minimizer $\W_v$ of volume $v<v_0$  with $\F(\W_v) \leq \kappa v^{\sfrac{(n-1)}{n}}$. Then $\W_v$ is connected and there is a point $x_0 \in M$ such that 
	\begin{equation}
		\label{eqn: qual hausdorff}
	d_{H,g}\Big(\partial \W_v,\, \exp_{x_0}(\partial(v^{\sfrac{1}{n}}K_{x_0}))\Big) <\frac{\eo v^{\sfrac{1}{n}}}{\beta_0} \qquad \mbox{and} \qquad \Big|\W_v \Delta \exp_{x_0}(v^{\sfrac{1}{n}}K_{x_0})\Big|_g <\frac{\eo v}{\beta_0} \,. 
	\end{equation}
	Here $K_{x_0}$ is the tangent Wulff shape at $x_0$ defined in \eqref{eqn: tangent wulff A}, and $$\beta_0(n,\kappa, F,\eo):= \frac{8 \C \eo}{\min\left\{ \Big(\frac{\m }{2\Lambda} n \omega_n^{\sfrac{1}{n}}\Big)^n, \frac{\e_0^n\, \omega_n }{2^{n+1}}\right\}}.
$$
\end{theorem}

\begin{remark}{\rm
In Theorem \ref{thm: uniform convergence3} we show that estimate \eqref{eqn: qual hausdorff} actually holds for any $\beta_0>0$, provided $v_0$ is also taken to be sufficiently small depending on $\beta_0$. However, the explicit choice of $\beta_0$ in the theorem statement provides a volume threshold under which local minimizers are connected.}
\end{remark}
\begin{remark}\label{rmk: true diameter}{\rm
Since $K_{x_0} \subset B_{g_{x_0}}(0, R')$ for a constant $R'>0$ depending only on $F$ and $g$,  Theorem~\ref{thm: uniform convergence2} implies that $\W_v$ satisfies a diameter bound $\Omega_v \subset B_{g}(x_0, R \, v^{\sfrac{1}{n}})$ 
where $R = R' + 2\e_0/\beta_0$ depends only on $F, g,$ and $n$. Moreover, \eqref{eqn: qual hausdorff} implies that
$$\W_v \Delta \exp_{x_0}(v^{\sfrac{1}{n}}K_{x_0}) \subset \Nb{g}{\partial \W_v}{{\eo v^{\sfrac{1}{n}}}/{\beta_0}}.$$
}
\end{remark}

We prove Theorem~\ref{thm: uniform convergence2} via  a compactness argument, using Theorem~\ref{thm: diameter bound} crucially at various points. To see the idea, take a sequence of volume-constrained $\eo$-local minimizers $\W_k$ of volume $v_k\to 0$. Using Theorem~\ref{thm: diameter bound}, we can pull back $\W_k$ in charts and show that,  after rescaling, the resulting sequence of sets $E_k$ in $\R^n$  subsequentially converges in $L^1$ to a limit set $E$ with unit volume.  Again using Theorem~\ref{thm: diameter bound}, we show in Section~\ref{ssec: no volume contraint} that the sets $\W_k$ satisfy scale-invariantly uniform density estimates. This upgrades the $L^1$ convergence of $E_k$ to Hausdorff convergence of the boundaries  and crucially allows us to deduce that the limit set $E$ is itself a volume-constrained local minimizer of a translation invariant anisotropic surface energy.  A scaling argument and the Alexandrov-type theorem \cite[Corollary 6.8]{DRKS} show that $E$ is a translation of the corresponding  Wulff shape. The translation invariance leads to  technical challenges bringing this statement back to $(M,g)$, which we tackle by comparing tangent Wulff shapes at different points using Proposition~\ref{prop: appendix tangent wulff}.

\subsection{Uniform quasi-minimality}\label{ssec: no volume contraint}
In this section, we prove that volume-constrained $\eo$-local minimizers of $\F$ with small volume satisfy a quasi-minimality property among non-volume-constrained competitors. In the language of \cite{Mag}, after rescaling, they satisfy a local version of being $(\Lambda ,r_0)$-minimizers of $\F$. Crucially,  the parameters $\Lambda$ and $r_0$ are independent of $v$.  Theorem~\ref{thm: diameter bound} is key in the proof, as it allows us to apply Lemma~\ref{lemma: dilation} to modify a local competitor into one with the prescribed volume while estimating the error in a uniform way. 
\begin{lemma}\label{lem: volume constraint}
 Let $(M,g)$ be a closed Riemannian manifold of dimension $n\geq 2$ and let $\F$ be an anisotropic surface energy on $M$. 
For each  $\eo>0$ and $\kappa>0$, there exist $v_0= v_0(n,\eo, \kappa, g)>0$ and $\Lambda=\Lambda(n,\kappa)>1$ such that the following holds. 
If $\W_v \subset M$ is a volume-constrained $\eo$-local minimizer of $\F$ with volume $v\in (0, v_0)$ and $\F(\W_v) \leq \kappa v^{\sfrac{(n-1)}{n}}$,  then $\W_v$ is a local $(\Lambda v^{-\sfrac{1}{n}} ,  \frac{\eo}{2}v^{\sfrac{1}{n}} )$-minimizer of $\F$, i.e. 
 \[
\F(\W_v) \leq \F(E) + \Lambda\,v^{-\sfrac{1}{n}} |\W_v\Delta E|_g. 
 \]
  for any set  $E\subset M$ such that $\W_v\Delta E \Subset \Nb{g}{ \partial \W_v}{\frac{\eo}{2}\, v^{\sfrac{1}{n}} }$.
 \end{lemma}

\begin{proof} Let $v_0$ and $\Lambda$ be  fixed constants to be specified in the proof and fix $v\in (0, v_0)$. Let $h = v^{-\sfrac{2}{n}}g$, so that $\W_v$ is a volume-constrained $\eo$-local minimizer of $\F_h$ with $|\W_v|_h=1$. In terms of the rescaled metric, we will prove that for any $E\subset M$ with $\W_v \Delta E \subset \Nb{h}{\partial \W_v}{\frac{\eo}{2}}$, we have 	
	\begin{equation}
		\label{eqn: h claim}
			\E_h(\W_v) \leq \E_h(E) \qquad \text{ where } \qquad
	\E_h(E) := \F_h(E) + \Lambda \left||E|_h -1\right|.
	\end{equation}
	Then, noting that $||E|_h -1| \leq |E\Delta \W_v|_h$ and scaling back to the original metric, this implies the lemma.
	
 To show \eqref{eqn: h claim}, it suffices to show that for any $E\subset M$ with $E \Delta \W_v \subset \Nb{h}{\partial \W_v}{\frac{\eo}{2}}$ and $\E_v(E) \leq 2\E_h(\W_v)$, we may find a set $\Emod$ with  
$$
\E_h (\Emod) \leq \E_h(E), \qquad
|\Emod|_h = 1, \qquad\Emod \Delta \W_v \subset \Nb{h}{\partial \W_v}{\eo}  ,$$
since then taking $\Emod$ as a competitor for the local minimality of $\W_v$, directly implies \eqref{eqn: h claim}.

So, fix $E\subset M$ with $E \Delta \W_v \Subset \Nb{h}{\partial \W_v}{\frac{\eo}{2}}$ and $\E_h(E) \leq 2\E_h(\W_v)$. Notice that 
\[
 \left||E|_h-1 \right| \leq \frac{\E_h(E)}{\Lambda} \leq \frac{2\E_h(\W_v)}{\Lambda}  = \frac{2\F_h(\W_v)}{\Lambda}  \leq \frac{2\kappa}{\Lambda}
\] 
Recall the dimensional constants $\ETA_0,c_n$  in Lemma~\ref{lemma: dilation}. Choose $\eta<\min\{\eta_0,\frac{\eo}{8c_n}\}$ and $\Lambda > 2\kappa/\ETA$.  According to Theorem~\ref{thm: diameter bound}, $\W_v \subset \cup_{j=1}^\JO B_h(x_i, 2)$ for points $x_1,\dots x_\JO \in M$, and thus we also have $\Nb{h}{E}{\frac{\eo}{2}} \subset \Nb{h}{\W_v}{\eo} \subset  \cup_{j=1}^\JO B_h(x_i, 4)$.
 Then, provided we choose $v_0$ small enough so \eqref{eqn: metrics close 2} holds with $D=4$, we can apply  Lemma~\ref{lemma: dilation} to obtain a set $\Emod$ that, thanks to our choice of $\eta$, satisfies
 \begin{align*}
	\Emod \Delta E \subset \Nb{h}{\pa E}{\frac{\eo}{2}}\subset \Nb{h}{\partial \W_v}{\eo}\qquad \text{ and }\qquad	|\Emod|_h=1,
\end{align*}
and
\begin{align*}
	 \F_h(\Emod) &\leq \F_h(E)(1 + C_n \left||E|_h-1\right|) \\
	 &\leq \F_h(E) + 2 \F(\W_v) C_n \left||E|_h-1\right| \leq  \F_h(E) + C_n\kappa \left||E|_h-1\right|. 
\end{align*}
Therefore, 
$\E_h(\Emod) = \F_h(\Emod)\leq  \F_h(E) +  C_n\kappa \left||E|_h-1\right| \leq \E_h(E)$
so long as $\Lambda > C_n \kappa.$ 
\end{proof}

\subsection{An intermediate form of Theorem~\ref{thm: uniform convergence2}}
Next, we prove a slightly weaker version of Theorem~\ref{thm: uniform convergence2}: In Theorem~\ref{thm: uniform convergence} below, a volume-constrained $\eo$-local minimizer $\W_v$  is shown to be close to a (projected via $\exp_{x_0}$) Wulff shape translated by some $y \in T_{x_0}M$.  The modulus of this translation, while tending to zero as $v\to 0$, could  be very large relative to the natural length scale $v^{\sfrac{1}{n}}$. In Section \ref{sec:conc} we will center to correct this translation error and prove that $\W_v$ is connected to complete the proof of Theorem~\ref{thm: uniform convergence2}.
\begin{theorem}\label{thm: uniform convergence}
	Fix a closed Riemannian $n$-manifold $(M,g)$ and an anisotropic surface energy $\F$ with integrand $F$. For every $\kappa>0$, $\beta>0$, $\eo>0$, and $\rho >0$, there exists $v_0  = v_0(g, F, \kappa, \eo, \beta, \rho)>0$ such that the following holds. 
	Let $\W_v$ be a volume-constrained $\eo$-local minimizer  of volume $v<v_0$  with $\F(\W_v) \leq \kappa v^{\sfrac{(n-1)}{n}}$. There are points $x \in M$ and $y \in B_{g_x}(0, \rho )\subset  T_xM$  such that 
	\begin{equation}
		\label{eqn: qual hausdorff1}
	d_{H,g}\big(\partial \W_v, \exp_{x}(\partial v^{\sfrac{1}{n}} K_{x} + y)\big) <\frac{\eo v^{\sfrac{1}{n}}}{\beta} \qquad \mbox{and} \qquad 
	\big|\W_v \Delta \exp_{x}(v^{\sfrac{1}{n}}K_{x} + y)\big|_g <\frac{\eo v}{\beta} \,. 
	\end{equation}
\end{theorem}

\begin{proof}
	We divide the proof in several steps:
	
{\it Step 0: Setup.}	
Supposing for the sake of contradiction that the statement is false, we find $r_0\leq \frac{\text{inj}_gM}{2}$ and a sequence of numbers $v_i \to 0$ and of volume-constrained $\eo$-local minimizers $\W_i$ of $\F$ with volume $v_i$ such that 
\begin{equation}
	\label{eqn: contra}
	d_{H,g}\big(\partial \W_i , \exp_x (\partial v_i^{\sfrac{1}{n}} K_x + y)\big) \geq \frac{\eo v_i^{1/n}}{\beta}\qquad \mbox{or} \qquad
	 \big|\W_i \Delta \exp_x(v_i^{\sfrac{1}{n}}K_x + y)\big|_g \geq \frac{\eo v_i}{\beta}
\end{equation}	
	for every $x \in M$ and $y \in B_{g_x}(0, r_0 )\subset  T_xM$. Let $v_0$ be chosen according to Theorem~\ref{thm: diameter bound}. Since  $v_i <v_0$ for $i$ large enough, we may apply Theorem~\ref{thm: diameter bound} to find a sequence of finite families of points 
 $\{x_{i,j}: i\in \N, j=1, \dots ,\JO\} \subset M$ such that $\W_i\subset \bigcup_{j=1}^{\JO} B_g(x_{i,j}, 2v^{\sfrac{1}{n}}_i)$. For  fixed $i \in \mathbb{N}$ and for each $j=1, \dots ,\JO$, let $\W_{i,j}$ be the union of all the connected components of $\W_i$ that  intersect $B_g(x_{i,j}, 2v^{\sfrac{1}{n}}_i)$ and do not intersect any of the previous balls $\{B_g(x_{i,k}, 2v^{\sfrac{1}{n}}_i)\}_{k=1}^{j-1}$. In particular, we observe that 
\begin{equation}\label{eqn:forseserve}
\W_{i,j}\subset B_g\big(x_{i,j}, 4\JO v^{\sfrac{1}{n}}_i\big) \quad \mbox{ for every $j=1,\dots, \JO$},
\end{equation}
 and $\W_{i,j}$ are pairwise disjoint in $j$.
 Since $M$ is compact, $x_{i,j} \to x_j \in M$ for every $j=1, \dots ,\JO$ after passing to a subsequence in $i$. For $i$ large enough, $B_g(x_{i,j}, 8 \JO v_i^{\sfrac{1}{n}} )\subset B_g(x_j, r_0)$ for every $j=1, \dots ,\JO$.\\

{\it Step 1: Pulling back and rescaling the problem.}
Fix an orthonormal basis $\{e_1, \dots, e_n\}$ for Euclidean space, and for each $j=1,\dots , \JO$, let $\psi_j: B_{g_{euc}}(0,r_0) \to B_g(x_j, r_0)$ be a normal coordinate map at $x_j$ and  let $z_j = 17J_0(j-1)e_1 \in \R^n$.
For large enough $i$, we can define the map
\[
\phi_{i,j} : B_{g_{euc}}\big(z_j, \,8\JO\big)\subset \R^n \to  M, \qquad \phi_{i,j}(x) := \psi_j \left(v_i^{\sfrac{1}{n}}(x-z_j) + \psi^{-1}_j(x_{i,j})\right),
\]
which first maps its domain to $B_{g_{euc}}(\psi^{-1}_j(x_{i,j}), 8\JO v_i^{\sfrac{1}{n}})$ homothetically, then maps this small ball to $M$ by the normal coordinate map.
Identifying the $\JO$ (a priori distinct)  copies of Euclidean space via the basis $\{e_1, \dots, e_n\}$, we view the balls $B_{g_{euc}}\big(z_j, \,8\JO\big)$ as disjoint subsets of the same Euclidean space.

 In particular, by \eqref{eqn:forseserve} we may define the pulled-back sets
 \begin{equation}\label{eqn:contain}
E_{i,j} := \phi_{i,j}^{-1}(\W_{i,j})\Subset  B_{g_{euc}}\big( z_j, 6\JO \big) \subset\R^n\,.
\end{equation}
We then let 
\[
E_i := \bigcup_{j=1}^{\JO} E_{i,j}, \quad \text{ with } \quad E_i \Subset X:=\bigcup_{j=1}^{\JO} B_{g_{euc}}\big(z_j, 6\JO\big) \subset\R^n\,.
\]
Observe that $X\subset B_{g_{euc}}(0,24\JO)$.

Let us define the rescaled metrics $h_i = v_i^{\sfrac{-2}{n}} g$, so that $|\W_i|_{h_i} =1$ and $\F_{h_i}(\W_i) \leq \kappa$ for all $i$. Up to passing to a subsequence with respect to $i$, we have
 \begin{equation}
 	\label{eqn: rough metric close}
 (1-1/i)g_{euc} \leq \phi_{i,j}^*h_i \leq (1+1/i)g_{euc} \quad \text{  on } B_{g_{euc}}\big( z_j, 8\JO \big)\qquad \forall i\in \N , \, j=1\dots,\JO\,.
 \end{equation}
Hence, by \eqref{eqn: rough metric close}, 
\begin{equation}
	\label{eqn: volumes converge}
||E_{i} |_{g_{euc}} - 1| \to 0.
\end{equation}

 Recall that the restriction $F_{x_j}(\cdot) := F(x_j, \cdot)$ of the anisotropic integrand defines a translation invariant surface energy $\bar{\F}_{x_j}$ on $T_{x_j}M$. 
  The normal coordinate map $\psi_j:B_{g_{euc}}(0,r_0) \to B_g(x_j, r_0) \subset M$ induces the linear map $L = (d\psi_j)_{|0}: \R^n \to T_{x_j} M$, which has $\det{L} =1$.  Let us denote by $\F^*_{x_j}$ the pulled-back tangent surface energy at $x_j$ by $L$, i.e. for a set $E\subset \R^n$ we let  
  \[
  \F^*_{x_j}(E)  = \bar{\F}_{x_j}(L(E)) = \int_{\partial^* E} F(x_j,L(\nu_E(y))\, d\Hi^{n-1}_{g_{euc}}(y).  
\]
Furthermore, for every set $E \subset X$ we define
$$
 \F^*(E) = \sum_{j=1}^\JO \F^*_{x_j}\big(E\cap B_{g_{euc}}(z_j, 8\JO)\big).
 $$
 It is easy to see that $\F^*$ can be extended to every subset of $\R^n$ and that $\F^*\equiv \F^*_{x_j}$ for subsets of $B_{g_{euc}}(z_j, 8\JO)$ as these balls are disjoint.
Let $\F^*_{i,j}(E) = \F_{h_i}(\phi_{i,j}(E))$ be the pulled-back $h_i$-surface energy of a set $E \subset B_{g_{euc}}(z_j, 8\JO)$ and let $\F^*_i(E) = \sum_{j=1}^\JO \F^*_{i,j}(E\cap B_{g_{euc}}(z_j, 8\JO))$ for every set $E \subset X$. By assumption,  $ \F^*_i(E_i) \leq \kappa$.   Moreover, by \eqref{eqn: rough metric close} and the continuity of $F$ with respect to $x$, we have 
 \begin{equation}
 	\label{eqn: F conv}
|  \F_{i,j}^*(E) - \F^*_{x_j}(E) | \leq \omega(i) \F_{i,j}^*(E) 
  \end{equation}
 for a modulus of continuity $\omega$ depending on $F$ and $g$.\\

{\it Step 2: Compactness in $L^1$.} Since $F(x ,\cdot ) \geq \m$, thanks to \eqref{eqn: F conv} and the assumption   $ \F^*_i(E_i) \leq \kappa$, we see that $P(E_i) \leq 2\m^{-1} \kappa$ for all $i$ sufficiently large. Moreover, using \eqref{eqn:contain}, we see that up to a subsequence,  
\begin{equation}\label{eqn: L1 conv}
E_{i,j} \to E_{0,j} \mbox{ in $L^1$ for sets of finite perimeter $E_{0,j}\subset B_{g_{euc}}(z_j, 7\JO)$}
\end{equation}
and we define
$$\bigcup_{j=1}^{\JO}E_{0,j} =: E\Subset X.$$  
By \eqref{eqn: volumes converge},  $|E|_{g_{euc}} =1$, and  by \eqref{eqn: F conv} and the lower semi-continuity of $\F_{x_j}^*$ with respect to $L^1$ convergence, 
\begin{equation}
	\label{eqn: lsc}
{\F}^*_{x_j}(E_{0,j}) \leq  \liminf_{i \to \infty } {\F}^*_{x_j}(E_{i,j}) = \liminf_{i \to \infty } {\F}^*_{i,j}(E_{i,j})\,.
\end{equation} 
\smallskip

{\it Step 3: Hausdorff convergence of the boundaries.} For $i$ sufficiently large and thus $v_i$ sufficiently small, we apply Lemma~\ref{lem: volume constraint}; after rescaling the metric and using \eqref{eqn: rough metric close}, we see that $E_i$ is a local $(2\Lambda, \, \frac{\eo}{4})$-minimizer of the energy $\F^*_i$ on $X$. A standard adaptation of the classical argument  (see for instance \cite[Theorem 21.11]{Mag}) shows that the sets $E_i$ enjoy uniform volume density estimates: there exist constants $c_0$ and $r_0$ depending only on $g,\m,\M, n, \Lambda$ and $\e_0$ such that for any $r< r_0$ and  $x \in \partial E_i$, 
\begin{equation}
	\label{eqn: density ests}
c_0 \leq \frac{ |E_i \cap B(x,r)|_{g_{euc}}}{ \omega_n r^n} \leq 1- c_0\,.
\end{equation}
The density estimates \eqref{eqn: density ests} let us improve $L^1$ convergence to Hausdorff convergence of the boundaries:
\begin{equation}
	\label{eqn: hausdoff conv 1}
d_{H,g_{euc}}(\partial E_i , \partial E) \to 0.
\end{equation}
Indeed, if \eqref{eqn: hausdoff conv 1} does not hold, then for some $\mathsf{r}>0$ and along an unrelabeled subsequence we have either:\\
 (a) a sequence of points $x_i \in \partial E_i$ such that $B_{g_{euc}}(x_i, \mathsf{r})\cap \partial E = \emptyset$ for all $i$, or else \\
 (b) a sequence of points $x_i \in \partial E$ such that $B_{g_{euc}}(x_i, \mathsf{r}) \cap \partial E_i = \emptyset$.
 
  In case (a), first suppose $B_{g_{euc}}(x_i, \mathsf{r}) \subset E$ for all $i$. The lower density estimate in \eqref{eqn: density ests} implies that
\[
|E\Delta E_i |_{g_{euc}} \geq |E\setminus E_i|_{g_{euc}}\geq |B(x_i, \mathsf{r}) \setminus E_i |_{g_{euc}} \geq c_0\mathsf{r}^n,
\]
contradicting the $L^1$ convergence. If instead $B_{g_{euc}}(x_i, \mathsf{r}) \subset E^c$ for all $i$, the same argument using the upper density estimate in \eqref{eqn: density ests} we again reach a contradiction.

In case (b) we argue differently since we do not yet know that $E$ satisfies density estimates. First suppose $B_{g_{euc}}(x_i ,  \mathsf{r}) \subset E_i^c$ for all $i$. By compactness, up to  a further subsequence, $x_i \to x \in \partial E$, and thus $B_{g_{euc}}(x,  \mathsf{r}/2) \subset E_i^c$ for all $i$ sufficiently large. So, $1_{E_i}(x)= 0$ for all $y \in B_{g_{euc}}(x,  \mathsf{r}/2).$ Since $1_{E_i}\to 1_E$ in $L^1(\R^n)$ and thus pointwise a.e., we see that $|E\cap B_{g_{euc}}(x, \mathsf{r}/2)|_{g_{euc}}=0$, contradicting \eqref{eqn: cleaned up sets}. The analogous argument leads to the same contradiction when instead $B_{g_{euc}}(x_i ,  \mathsf{r}) \subset E_i$ for all $i$. This proves \eqref{eqn: hausdoff conv 1}.\\

{\it Step 4: $E$ is a local minimizer of ${\F}^*$.}  Next, we claim that $E$ is a volume-constrained $\frac{\eo}{4}$-local minimizer of the energy ${\F}^*$ in $\R^n$. To this end, take a set $G\subset \R^n$ with $|G|_{g_{euc}}=1$ and $E\Delta G \subset \Nb{g_{euc}}{\pa E}{\frac{\eo}{4}}$. Thanks to \eqref{eqn: hausdoff conv 1}, we also have  $E\Delta E_i \subset \Nb{g_{euc}}{\partial E_i}{\frac{\eo}{4}}$ for all $i$ sufficiently large. In turn, by the triangle inequality property of the symmetric difference \eqref{eqn: triangle}, we have 
\begin{align*}
	E_i\Delta G \subset (E_i\Delta E) \cup (E\Delta G) \subset \Nb{g_{euc}}{\partial E_i}{\frac{\eo}{4}}\cup \, \Nb{g_{euc}}{\partial E}{\frac{\eo}{4}} \subset \Nb{g_{euc}}{\partial E_i}{\frac{\eo}{2}}.
\end{align*}
Letting $G_j:=G\cap B(z_j, 8\JO)$, note
that $G_j\subset B(z_j, 7\JO)$ for every $j=1,\dots,\JO$. Hence, we can define $\hat{G}_{i,j} = \phi_{i,j}(G_j) \subset M$ and $\hat{G}_i:=\cup_{j=1}^\JO \hat{G}_{i,j}$. Letting $h_i=v_i^{\sfrac{-2}{n}}g$,  we see that $\W_i \Delta \hat{G}_{i} \subset \Nb{h_i}{\pa \W_i}{{\eo}}$ and thanks to \eqref{eqn: rough metric close}, up to passing to a subsequence,  $||\hat{G}_i|_{h_i}-1| < 1/i$. 
Thus, applying Lemma~\ref{lemma: dilation}, we obtain sets $\tilde{G}_i \subset M$ with $|\tilde{G}_i|_{h_i}=1$, i.e. $|\tilde{G}_i|_g = v_i$ and $\F(\tilde{G}_i) \leq (1 +1/i) \F(\hat{G}_{i})$. In particular, $\tilde{G}_i$ is an admissible competitor for the minimality of $\W_i$, i.e. $\F(\W_i) \leq \F(\tilde{G}_i) \leq (1+1/i)\F(\hat{G}_{i})$. 
Pulling the sets and energies back in charts and applying \eqref{eqn: F conv}, we find that 
\[
\F_i^*(E_i) \leq (1+1/i) \F_i^*(G) \leq  (1+2/i) \F^*(G).
\]
Taking the limit infimum and recalling \eqref{eqn: lsc}, we conclude that $\F^*(E) \leq \F^*(G)$, proving the claim.\\

{

{\it Step 5: $E$ is a Wulff shape for ${\F}^*_{x_1}$.}
Through the choice of basis for $T_{x_j}M$ via the normal coordinate map, we have identified $T_{x_j}M$ with $\R^n$ and therefore may identify the volume-$1$ tangent Wulff shape
$K_{x_j}\subset T_{x_j}M$ with a subset of $\R^n,$ which we again denote by $K_{x_j}\subset \R^n$, that is the (Euclidean) unit-volume Wulff shape for the translation invariant surface energy ${\F}^*_{x_j}$ on $\R^n$ defined above.

Since the balls $B_{g_{euc}}(z_j, 8\JO)$ are disjoint,  the set $E \cap B_{g_{euc}}(z_j, 8\JO)$ is a set of finite perimeter that is a local minimizer of the smooth, uniformly elliptic, translation invariant anisotropic surface energy ${\F}^*_{x_j}$.  According to the Alexandrov-type theorem of the first author, Kolasi\'{n}ski, and Santilli \cite[Corollary 6.8]{DRKS}, we deduce that $E \cap B_{g_{euc}}(z_j, 8\JO)$ is the union of finitely many Wulff shapes for ${\F}^*_{x_j}$ with equal volume, which are either disjoint or tangent to each other. We remark that in particular these Wulff shapes have boundaries at least of class $C^1$.

We now claim that $E$ has just one connected component. We will prove this by contradiction.
Assume without loss of generality that $E$ has two connected components $E_1$ and $E_2$, and that $1=|E|_{g_{euc}}=|E_1|_{g_{euc}}+|E_2|_{g_{euc}}$. 

First we observe that $E_1$ and $E_2$ cannot be tangent to each other, as otherwise there exists a tangent point $x\in \partial E_1\cap \partial E_2$ such that for every $\gamma>0$ there exists a radius $r>0$ for which 
$\Hi^{n-1}_{g_{euc}}( \partial B(x,r)\setminus (E_1\cup E_2))\leq \gamma r^{n-1}$ and $B(x,r)\subset  \Nb{g_{euc}}{\pa E}{\frac{\eo}{4}}$. Hence
$ E\cup B(x,r)$ decreases the energy $\F^*$ and increases the volume, so that a suitable rescaling is an admissible competitor and violates the local minimality of $E$.

Denote by $\alpha_i:=\F^*(E_i)$ and $v_i:=|E_i|$ for $i=1,2$. 
For $t$ small enough, we define
$$E^t:=(1+t)E_1\cup g(t)E_2 \subset X,$$
where $g(t)$ is defined by the constraint of volume $(1+t)^nv_1+g(t)^nv_2=v_1+v_2$. Simple algebraic manipulations of this volume constraint shows that
$$g(t)=\Big(1-\left((1+t)^n-1\right)\frac{v_1}{v_2}\Big)^{\sfrac 1n},$$
from which we compute
$$g'(0)=-\frac{v_1}{v_2}, \qquad g''(0)=(1-n)\frac{v_1}{v_2} \Big(1+\frac{v_1}{v_2}\Big).$$
We use these values to compute the derivatives of
$$
f(t):=\F^*(E^t)=(1+t)^{n-1}\alpha_1+g(t)^{n-1}\alpha_2.
$$
Via simple calculus and the local minimality of $E$, we compute
$$0=f'(0)=(n-1)\Big(\alpha_1-\frac{v_1}{v_2}\alpha_2\Big),$$
which is equivalent to $\alpha_1/\alpha_2=v_1/v_2$. Using this equality in computing $f''(0)$, and again by local minimality of $E$, we obtain the following contradiction
$$0\leq f''(0)=-\alpha_1(n-1)\Big(1+\frac{v_1}{v_2}\Big)<0.$$

We conclude that $E$ has only one connected component. In particular, $E$ must be contained in only one of the balls $B_{g_{euc}}(z_j, 8\JO)$.
We will assume, without loss of generality, that $E\subset B_{g_{euc}}(0, 8\JO)$.

In particular, we conclude that $E$ comprises exactly one Wulff shape, i.e. 
\begin{equation}
	\label{eqn: E is a wulff shape}
	E = K_{x_1} +y\qquad \mbox{ for some $y \in \R^n$. }
\end{equation}
 Moreover, keeping in mind that $0 \in K_{x_1}$  and $E\subset B_{g_{euc}}(0, 8 \JO )$, we see that $|y|_{g_{euc}} \leq 8\JO$. \\
 
}

{\it Step 6: Contradiction to the initial claim.}
Together,  \eqref{eqn: L1 conv}, \eqref{eqn: hausdoff conv 1} and \eqref{eqn: E is a wulff shape} show that for $i$ sufficiently large, 
\begin{equation}\label{convergences}
d_{H,{g_{euc}}} \big(\partial E_i , \partial K_{x_1} +y \big) <\eo/\beta \qquad \mbox{and} \qquad \big|E_i \Delta(  K_{x_1}+y) \big|_{g_{euc}} <\eo/\beta.
\end{equation}
Mapping these sets onto $M$ by $\phi_{i,1}$, \eqref{convergences} implies that 
\[
d_{H,g} \big(\partial \W_i , \exp_{x_1} \big(\partial (  v_i^{\sfrac{1}{n}}K_{x_1}) + y_i \big)\big)< \frac{\eo \, v_i^{\sfrac{1}{n}}}{\beta}, \qquad \mbox{and} \qquad \big|\W_i \Delta \exp_{x_1} (v_i^{\sfrac{1}{n}}K_{x_1}  + y_i ) \big|_g < \frac{\eo v_i}{\beta}.
\]
By Step 1, $y_i = v_i^{\sfrac{1}{ n}}y + \psi_1^{-1} (x_{i,1}) \in B_{g_{x_1}}(0, r_0)$.  This contradicts \eqref{eqn: contra} and completes the proof.
\end{proof}

\subsection{Recentering.}\label{sec:conc}
We now improve Theorem~\ref{thm: uniform convergence}, simply recentering our parametrization to correct the translation $y$, by means of Proposition \ref{prop: appendix tangent wulff}, to obtain the following:

\begin{theorem}\label{thm: uniform convergence3}
	Fix a closed Riemannian $n$-manifold $(M,g)$ and an anisotropic surface energy $\F$ with integrand $F$. For every $\kappa>0$, $\eo>0$, and $\beta>0$, there exists $v_0  = v_0(g, F, \kappa, \eo, \beta)>0$ such that the following holds. 
	Let $\W_v$ be a volume-constrained $\eo$-local minimizer  of volume $v<v_0$  with $\F(\W_v) \leq \kappa v^{\sfrac{(n-1)}{n}}$. There is a point $x \in M$ such that 
	\begin{equation}
		\label{eqn: qual hausdorff2}
	d_{H,g}\Big(\partial \W_v, \exp_{x}(\partial(v^{\sfrac{1}{n}}K_{x})\Big) <\frac{\eo v^{\sfrac{1}{n}}}{\beta} \qquad \mbox{and} \qquad \Big|\W_v \Delta \exp_{x}(v^{\sfrac{1}{n}}K_{x})\Big|_g <\frac{\eo v}{\beta} \,. 
	\end{equation}
\end{theorem}

\begin{proof}[Proof of Theorem~\ref{thm: uniform convergence3}] 
We apply Theorem~\ref{thm: uniform convergence} with $\kappa=\kappa$, $\eo=\eo$, $\beta=2\beta$ and $\rho= \frac{\eo}{2C\beta}$. Notice that $v_0$ will now depend only on $g, F, \kappa, \eo, \beta$, as $\rho$ depends just on $\eo$ and $\beta$.
We deduce the validity of \eqref{eqn: qual hausdorff1}. We now apply Proposition \ref{prop: appendix tangent wulff} choosing $x_0=x$, $x_1=\exp_x(y)$, $z_1=y$, $\rho=\rho$, $r=v^{\sfrac{1}{n}}$, to deduce \eqref{eqn: qual hausdorff2}.
\end{proof}

\subsection{Conclusion of the proof.}\label{sec:conc}
Finally, we can conclude the proof of Theorem~\ref{thm: uniform convergence2}.  We need only to show that a volume-constrained $\eo$-local minimizer is connected. 

\begin{proof}[Proof of Theorem~\ref{thm: uniform convergence2}] 
Choose $\beta=\beta_0$ and let $v_0$ be as in Theorem~\ref{thm: uniform convergence3}.
Since $\beta_0$ depends just on $n,\kappa, F,\eo$, we deduce that $v_0$ depends only on $n,g,\kappa, F,\eo$. Applying Theorem~\ref{thm: uniform convergence3}, we deduce the validity of \eqref{eqn: qual hausdorff}.
 We are now left to prove that $\W_v$ has one connected component.
To prove this, we first claim that the volume of every connected component $U$ of $\W_v$ satisfies the following lower bound: 
\begin{equation}\label{boundd}
	|U|_g\geq \min\left\{ \Big(\frac{\m }{2\Lambda} n \omega_n^{\sfrac{1}{n}}\Big)^n, \frac{\e_0^n\, \omega_n }{2^{n+1}}\right\}  v.
\end{equation}
	Otherwise, denote by $U$ a connected component such that
	\begin{equation}\label{contrboundd}
	|U|_g<\min\left\{ \Big(\frac{\m }{2\Lambda} n \omega_n^{\sfrac{1}{n}}\Big)^n, \frac{\e_0^n\, \omega_n }{2^{n+1}}\right\} v.
	\end{equation}
	We wish to apply the local quasi-minimality (without volume constraint) from Lemma \ref{lem: volume constraint} with the competitor $\W_v\setminus U$. To this end, we  argue along the lines of Lemma~\ref{lem: containment} to see that $U \subset  \Nb{g}{\partial U}{\frac{\eo}{2}v^{\sfrac{1}{n}}}$: if $x \in U \setminus \Nb{g}{\partial U}{\frac{\eo}{2}v^{\sfrac{1}{n}}}$, then $B_g(x , \frac{\eo}{2}v^{\sfrac{1}{n}} )\subset U$ and $|U|_g \geq \eo^n\omega_nv/2^{n+1}$ (provided $v_0$ is sufficiently small in terms of $g$ and $n$), contradicting \eqref{contrboundd}.
So, by Lemma \ref{lem: volume constraint} we have that
$$ \F(\W_v\setminus U) + \Lambda\,v^{-\sfrac{1}{n}} |\W_v\Delta (\W_v\setminus U)|_g\geq\F(\W_v)= \F(U)+\F(\W_v\setminus U)\overset{\eqref{eqn: Wulff small volume}}{\geq} \Big(\frac{\m }{2} n \omega_n^{\sfrac{1}{n}}\Big) |U|_g^{\sfrac{(n-1)}n}+\F(\W_v\setminus U),$$
from which we deduce the following contradiction:
$$\frac{\m }{2} n \omega_n^{1/n}\overset{\eqref{contrboundd}}{>} \Lambda\,v^{-\sfrac{1}{n}}|U|_g^{\sfrac 1n}\geq \frac{\m }{2} n \omega_n^{\sfrac{1}{n}}.$$
Since \eqref{contrboundd} leads to contradiction, we deduce the validity of \eqref{boundd}. 

We also observe that, by \eqref{eqn: qual hausdorff}, for every connected component $U$ of $\W_v$:
$$\partial U\subset \partial \W_v \subset  \Nb{g}{\exp_x(\partial(v^{\sfrac{1}{n}}K_x) + y)}{{\eo v^{\sfrac{1}{n}}}/{\beta_0}}.
 $$ 
We deduce that either 
\begin{itemize}
\item[(Case 1)]  
$ \exp_x(v^{\sfrac{1}{n}}K_x + y)\setminus \Nb{g}{\exp_x(\partial(v^{\sfrac{1}{n}}K_x) + y}{\frac{\eo v^{1/n}}{\beta_0}}\subset U$, or 
\item[(Case 2)] $U\subset 
\Nb{g}{\exp_x(\partial(v^{\sfrac{1}{n}}K_x)+ y)}{\frac{\eo v^{1/n}}{\beta_0}}$.
\end{itemize}
Given the $L^1$ estimate in \eqref{eqn: qual hausdorff1}, there can be only one connected component satisfying (Case 1). Moreover, \eqref{boundd} implies that no connected component can satisfy (Case 2), because (Case 2) provides the following volume upper bound for $U$:
$$
\min\left\{ \Big(\frac{\m }{2\Lambda} n \omega_n^{\sfrac{1}{n}}\Big)^n, \frac{\e_0^n\, \omega_n }{2^{n+1}}\right\} v\leq |U|_g\leq 4 {\frac{\eo v^{\sfrac{1}{n}}}{\beta_0}} P(v^{\sfrac{1}{n}}K_x + y)\overset{\eqref{eqn:massimo}}{\leq}
	4 {\frac{\eo v^{\sfrac{1}{n}}}{\beta_0}} \C v^{\sfrac{(n-1)}{n}}={\frac{4 \C \eo}{\beta_0}}  v
$$
which contradicts the definition of 
$$\beta_0(n,\kappa, F,\eo):= \frac{8 \C \eo}{\min\left\{ \Big(\frac{\m }{2\Lambda} n \omega_n^{\sfrac{1}{n}}\Big)^n, \frac{\e_0^n\, \omega_n }{2^{n+1}}\right\}}.
$$
 We conclude that $\W_v$ has just one connected component.
\end{proof}

\section{Quantitative closeness to a Wulff shape}\label{sec: quantitative}
In this section we complete the proof of Theorem~\ref{thm: local min}. To begin, in Corollary~\ref{cor: improved estimate}, we prove a quantitative version of Theorem~\ref{thm: uniform convergence2} through an application of Figalli-Maggi-Pratelli's quantitative Wulff inequality \cite{FiMaPr}.  This application originates with \cite{Figalli-Maggi-drops} in the context of {\it global} minimizers of a related anisotropic variational problem on Euclidean space.  A fundamental difference in the present setting of local minimizers is that it was essential to first prove the qualitative Theorem~\ref{thm: uniform convergence2} in order to use a (projected) tangent  Wulff shape as a competitor for local minimality.
\begin{corollary}\label{cor: improved estimate}
	Fix a Riemannian $n$-manifold $(M,g)$ and an anisotropic surface energy $\F$. 
	For every $\kappa>0$ and $\eo>0$, there exists $v_0  = v_0(g, F, \kappa, \eo)$ and $C(g, F, \kappa, \eo)$ such that the following holds. Let $\W_v$ be a volume-constrained $\eo$-local minimizer $\W_v$ of volume $v<v_0$ with $\F(\W_v) \leq \kappa v^{\sfrac{(n-1)}{n}}$. Then $\W_v$ has one connected component and there is a point $x_0 \in M$ such that 
	\begin{equation}
		\label{eqn: quant hausdorff}
	\frac{d_{H,g}\Big(\partial \W_v, \exp_{x_0}(\partial v^{\sfrac{1}{n}}K_{x_0})\Big)}{v^{\sfrac{1}{n}}} < C v^{\sfrac{1}{2n^2}} \qquad \mbox{and} \qquad \frac{\Big|\W_v \Delta \exp_{x_0}(v^{\sfrac{1}{n}}K_{x_0})\Big|_g}{v} < Cv^{\sfrac{1}{2n}} \,. 
	\end{equation}
\end{corollary}

\begin{proof}
Let $v_0, \beta_0$ and $x_0\in M$ be as in Theorem~\ref{thm: uniform convergence2}, and let  $\psi = \exp_{x_0}$. 
To lighten notation, we let $K= K_{x_0}$ and $\bar{\F} = \bar{\F}_{x_0}$.
  By Remark~\ref{rmk: true diameter} we may define $G_v = \psi^{-1}(\W_v) \subset B_{g_{x_0}}(0, Rv^{\sfrac{1}{n}})$ where $R = R(n,g,F)>0$. On $B_{g_{x_0}}(0, 10 R)$ we have
\begin{equation}\label{eqn: error in the metrics}
 (1- cv^{\sfrac{2}{n}})g_{x_0} \leq \psi^* g \leq (1+ cv^{\sfrac{2}{n}})g_{x_0}
\end{equation}
 for a constant $c=c(g)$, and so \eqref{eqn: error in the metrics} 
\begin{equation}\label{eqn: error in volumes}
\big| |G_v|_{g_{x_0}} -v \big| \leq cv^{1 + \sfrac{2}{n}}\qquad \mbox{and} \qquad\big| |\psi(v^{\sfrac{1}{n}}K)|_g - v\big| \leq cv^{1+ \sfrac{2}{n}}\,.
\end{equation} 
So, just as in the proof of Lemma~\ref{lemma: dilation} we can choose a dilation factor $\lambda>0$ with $|\lambda -1|\leq c v^{\sfrac{2}{n}}$ such that the set $E := \psi (\lambda v^{\sfrac{1}{n}} K)$ has $|E|_g=|\W_v|_g$. Moreover, up to further decreasing $v_0$ depending on $g, F, n$ and applying \eqref{eqn: triangle} and Theorem~\ref{thm: uniform convergence2}, we have 
$\W_v \Delta E    \subset \Nb{g}{\partial \W_v}{{\eo v^{\sfrac{1}{n}}}}.$
So, $E$ is an admissible competitor for the local minimality of  $\W_v$ and thus $\F(\W_v)\leq \F(E)$.  On the other hand, thanks to \eqref{eqn: error in the metrics} and \eqref{eqn: F C1 reg}, 
\begin{align*}
	\F(E) & \leq \big(1+c v^{\sfrac{1}{n}}\big)\bar{\F}(\lambda v^{\sfrac{1}{n} }K) \\&=\big(1+c v^{\sfrac{1}{n}}\big) \lambda^{n-1}v^{\sfrac{(n-1)}{n}} \bar{\F}(K)
 \leq \big(1+c v^{\sfrac{1}{n}}\big) v^{\sfrac{(n-1)}{n}}\bar{\F}(K),
	\end{align*}
	and $\F(\W_v) \geq 	(1-c v^{\sfrac{1}{n}})\bar{\F}(G_v)$
	for $c = c(F, g)$. 
Together these yield
$
\bar{\F}(G_v) \leq (1 + c v^{\sfrac{1}{n}})v^{\sfrac{(n-1)}{n}} \bar{\F}(K) .$
In particular, additionally using \eqref{eqn: error in volumes}, we estimates the scale-invariant deficit in the Wulff inequality \eqref{Wulffineq}:
$$\delta(G_v)
=\frac{\bar{\F}(G_v) }{|G_v|_{g_{x_0}}^{\frac{n-1}{n}} \bar{\F}(K)} -1\leq Cv^{\sfrac{1}{n}}$$
for $C= C(n,g, F)$.
By the quantitative  Wulff inequality \cite{FiMaPr}, there  exists $y\in \R^n$ such that 
$$
\frac{|G_v \Delta (y+v^{\sfrac{1}{n}}K)|_{g_{x_0}}^2}{v^2}\leq  C \delta(G_v) \leq Cv^{\sfrac{1}{n}}.
$$
Moreover,  $|y|\leq Cv^{\sfrac{1}{n}}$ \eqref{eqn: qual hausdorff}. Using the scale-invariantly uniform density estimates for $\W_v$ as  in Step 3 of Theorem~\ref{thm: uniform convergence}, we obtain the estimate 
$${d_{H,g}(\partial \W_v,\psi(\partial(y+ v^{\sfrac{1}{n}}K))}^n \leq C  |\W_v \Delta\psi(\partial(y+ v^{\sfrac{1}{n}}K))|_g < C v^{1 + \sfrac{1}{2n}}.$$

It remains to eliminate that translation $y$. To this end we apply Proposition \ref{prop: appendix tangent wulff} choosing $x_0={x_0}$, $x_1=\exp_{x_0}(y)$, $z_1=y$, $\rho=Cv^{\sfrac{1}{n}}$, $r=v^{\sfrac{1}{n}}$. Since $\rho r\leq C v^{\sfrac 1n} v^{\sfrac 1n}\leq C v^{\sfrac 1n} v^{\sfrac{1}{2n^2}}$, we obtain:
$$
\frac{|\W_v\Delta \exp_{x_1}(v^{\sfrac{1}{n}}K_{x_1})|_g^2}{v^2}\leq Cv^{\sfrac{1}{n}}, \qquad \mbox{and} \qquad 
\frac{d_{H,g}(\partial \W_v, \exp_{x_1}(\partial(v^{\sfrac{1}{n}}K_{x_1}))}{v^{\sfrac{1}{n}}} < C v^{\sfrac{1}{2n^2}}.$$
Up to relabelling the point $x_1$, this is exactly the desired claim \eqref{eqn: quant hausdorff}.
\end{proof}

Finally, we conclude the proof of Theorem~\ref{thm: local min}.
\begin{proof}[Proof of Theorem \ref{thm: local min}]
Let $v_0$  and $x_0$ be as in Corollary~\ref{cor: improved estimate}.
The quantitative estimate \eqref{eqn: quant hausdorff} and the assumed regularity of $F$ can be used as in the proof of \cite[Theorem 2]{Figalli-Maggi-drops} to show that $\W_v$ is of class $C^{2,\alpha}$ and to obtain the following almost anisotropic umbilicality estimate on the anisotropic second fundamental form of $\partial \exp_{x_0}^{-1}(\W_v)$:
\begin{equation}\label{umbilical}
\|D^2F_{x_0}(\nu_{\exp_{x_0}^{-1}(\W_v)})D\nu_{\exp_{x_0}^{-1}(\W_v)} - \mbox{Id}\|_{C^0(\partial \exp_{x_0}^{-1}(\W_v))}\leq C(g, F, \kappa, \eo,\alpha)v^{\frac{2\alpha}{n+2\alpha}}.
\end{equation}
Anisotropic almost umbilical surfaces enjoy higher order quantitative closeness to the Wulff shapes, see for instance \cite{DG,DG2}.

Since the proof of \eqref{umbilical} is obtained repeating the arguments that are laid out in detail in \cite{Figalli-Maggi-drops}, we simply outline the steps and highlight the differences: 

\begin{itemize}
\item The first step---analogous to \cite[Theorem 8]{Figalli-Maggi-drops}---is to show that $\exp_{x_0}^{-1}(\partial \W_v)\subset T_{x_0}M$ is locally the graph of a $C^{1,\alpha}$ function over an affine hyperplane in $T_{x_0}M$. The idea is to  show the hypotheses of the $\e$-regularity theorem (\cite{Almgren1, Bomb, SSA, SchoenSimon, DS}) are satisfied at each point on $\exp_{x_0}^{-1}(\partial \W_v)$  and at a small enough scale by constructing a competitor from the projection of $\exp_{x_0}^{-1}(\partial \W_v)$ locally onto a disk.
\item The proof of \cite[Theorem 8]{Figalli-Maggi-drops} goes through in a nearly identical fashion in the present context with the following two substituted ingredients: use Corollary~\ref{cor: improved estimate} in place of \cite[Corollary 1]{Figalli-Maggi-drops} to show that $\W_v$ is uniformly close to a Wulff shape, and use Lemma~\ref{lem: volume constraint} and \eqref{eqn: Wulff small volume} in place of \cite[Lemma 9]{Figalli-Maggi-drops} to obtain the non-volume-constrained quasi-minimality property akin to \cite[Eqn. (C.162)]{Figalli-Maggi-drops} (which is applied only to local competitors as in Lemma~\ref{lem: volume constraint}). 
\item The next step is to used Schauder estimates to improve the local estimates to $C^{2,\alpha}$ estimates for the functions locally parametrizing $\exp_{x_0}^{-1}(\partial \W_v)$, as in \cite[Theorem 12]{Figalli-Maggi-drops}. Here we use the assumption that $F$ is $C^{2,\alpha}$ in $x$ and $\nu$. 
\item Finally \eqref{umbilical} follows from interpolating between the Hausdorff estimates of Corollary~\ref{cor: improved estimate} and the $C^{2,\alpha}$ estimates arguing just as in \cite[Theorem 13]{Figalli-Maggi-drops}.
\end{itemize}
We are left to prove that $\W_v$ is geodesically convex.
Let $\psi$ be the normal coordinate map at $x_0$ and observe that by \eqref{eqn: quant hausdorff} there exists $\eta(g, F, \kappa, \eo)>0$ such that
\begin{equation}\label{palladentro}
B_{g_{euc}}(0,\eta v^{\frac 1n}) \Subset E_v:=\psi^{-1}(\W_v)\subset \R^n
\end{equation}  
and that by \eqref{umbilical} $E_v$ is uniformly convex. In particular, we have the following lower bound on the smallest eigenvalue $\lambda_1$ of the second fundamental form of $E_v$:
\begin{equation}\label{eqn: lower bound ppl curvature}
	\lambda_1\geq C(F)v^{-\frac 1n}.
\end{equation}
This implies that, given two points $a,b\in \partial E_v$, any curve connecting $a,b$ that  is contained in $\R^n\setminus E_v$ has length at least $|b-a|+C(F)v^{\frac 1n}\frac{|b-a|^2}{2}$. 

Assume by way of contradiction that $\W_v$ is not geodesically convex and so there exists a minimizing geodesic $\tilde{\gamma}:[0,\ell]\to M$ parametrized by arclength with end points $\tilde a,\tilde b \in \partial \W_v$ and such that $\tilde{\gamma}(0,\ell)\subset M\setminus \overline{\W}_v$. Let  $a:=\psi^{-1}(\tilde a)$, $b:=\psi^{-1}(\tilde b)$, and $\gamma = \psi^{-1} \circ \tilde{\gamma }:[0,\ell] \to  \R^n$. 

If $|a-b|\geq \eta v^{\sfrac{1}{n}}/4$, then the observation above, together with \eqref{palladentro}, shows that the image through $\psi$ of the segment $[a,b]$ has smaller length than any curve contained in $M\setminus \W_v$ joining $\tilde{a}$ with $\tilde{b}$ provided  $v_0$ (and hence $v$) is chosen sufficiently small depending only on $g, F, \eo$, and  $\kappa$. 

Thus we must have  $|a-b|<\eta v^{\sfrac 1n}/4$. Let  $z= (a+b)/2$ be  the midpoint and for $t \in [0,1]$ set  $a_t:= a-tz$, $b_t := b-tz$ and $\Gamma_t = Im({\gamma}) - tz$. We have $\Gamma_0 \subset B_{g_{euc}}(z,\eta v^{\sfrac 1n})$ provided  again $v_0$ is chosen sufficiently small depending only on $g, F, \eo$, and  $\kappa$, otherwise the straight-line competitor violates the minimality of $\gamma$. So,  $\Gamma_1 \subset B_{g_{euc}}(z,\eta v^{\sfrac 1n})\Subset  E_v$. Furthermore, by strict convexity, $a_t, b_t\in E_v$ for all $t \in (0,1]$.   So, $s_0 := \inf\{ s \in [0,1] : \Gamma_t \subset E_v\}$ lies in $(0,1)$ and there exists $t_0\in (0,\ell)$ such that $\gamma(t_0) - s_0 z \in \Gamma_{s_0}\cap \partial E_v$. This implies that at  corresponding point $\gamma(t_0)$, the curvature of $\gamma$ with respect to ${g_{euc}}$ is at  least $C(F) v^{-\sfrac{1}{n}}$ by \eqref{eqn: lower bound ppl curvature}. On the other hand, estimating the metric coefficients $g_{ij}$ and Christoffel symbols $\Gamma_{ij}^k$, we see that for $v_0$ sufficiently small this contradicts the  fact that ${\gamma}$ satisfies the geodesic equation $\frac{d^2}{dt^2} {\gamma}^k +\Gamma_{ij}^k \frac{d}{dt}{\gamma}^i \frac{d}{dt}{\gamma}^j  = 0$.
Hence  $\W_v$ is geodesically convex.
 \end{proof}
%
%
%
%
%
%
\bibliography{RefsLocalMin.bib}

\begin{thebibliography}{DMMN18}

\bibitem[All74]{allard1974characterization}
W.~K. Allard.
\newblock A characterization of the area integrand.
\newblock In {\em Symposia Mathematica}, volume~14, pages 429--444, 1974.

\bibitem[Alm68]{Almgren1}
F.~J. Almgren, Jr.
\newblock Existence and regularity almost everywhere of solutions to elliptic
  variational problems among surfaces of varying topological type and
  singularity structure.
\newblock {\em Ann. of Math. (2)}, 87:321--391, 1968.

\bibitem[Alm76]{Almgren}
F.~J. Almgren, Jr.
\newblock Existence and regularity almost everywhere of solutions to elliptic
  variational problems with constraints.
\newblock {\em Mem. Amer. Math. Soc.}, 4(165):viii+199, 1976.

\bibitem[BC14]{BC}
M.~Bonacini and R.~Cristoferi.
\newblock Local and global minimality results for a nonlocal isoperimetric
  problem on {$\Bbb{R}^N$}.
\newblock {\em SIAM J. Math. Anal.}, 46(4):2310--2349, 2014.

\bibitem[Bel]{BellCapillary}
C.~Bellettini.
\newblock Embeddedness of liquid-vapour interfaces in stable equilibrium.
\newblock {\em Preprint at ar{X}iv:2104.12198}.

\bibitem[BM94]{bromo}
J.~E. Brothers and F.~Morgan.
\newblock {The isoperimetric theorem for general integrands.}
\newblock {\em Michigan Math. J.}, 41(3):419 -- 431, 1994.

\bibitem[Bom82]{Bomb}
E.~Bombieri.
\newblock Regularity theory for almost minimal currents.
\newblock {\em Arch. Rational Mech. Anal.}, 78(2):99--130, 1982.

\bibitem[CFMN18]{Nonlocal}
G.~Ciraolo, A.~Figalli, F.~Maggi, and M.~Novaga.
\newblock Rigidity and sharp stability estimates for hypersurfaces with
  constant and almost-constant nonlocal mean curvature.
\newblock {\em J. Reine Angew. Math.}, 741:275--294, 2018.

\bibitem[CGOS18]{CGOS}
S.~Conti, M.~Goldman, F.~Otto, and S.~Serfaty.
\newblock A branched transport limit of the {G}inzburg-{L}andau functional.
\newblock {\em J. \'{E}c. polytech. Math.}, 5:317--375, 2018.

\bibitem[CL12]{CiLe}
M.~Cicalese and G.~P. Leonardi.
\newblock A selection principle for the sharp quantitative isoperimetric
  inequality.
\newblock {\em Arch. Ration. Mech. Anal.}, 206(2):617--643, 2012.

\bibitem[CM17]{CM}
G.~Ciraolo and F.~Maggi.
\newblock On the shape of compact hypersurfaces with almost-constant mean
  curvature.
\newblock {\em Comm. Pure Appl. Math.}, 70(4):665--716, 2017.

\bibitem[CTG21]{candautilh2021existence}
J.~Candau-Tilh and M.~Goldman.
\newblock Existence and stability results for an isoperimetric problem with a
  non-local interaction of {W}asserstein type.
\newblock {\em Preprint at arXiv:2108.11102}, 2021.

\bibitem[DMMN18]{DMMN}
M.~G. Delgadino, F.~Maggi, C.~Mihaila, and R.~Neumayer.
\newblock Bubbling with {$L^2$}-almost constant mean curvature and an
  {A}lexandrov-type theorem for crystals.
\newblock {\em Arch. Ration. Mech. Anal.}, 230(3):1131--1177, 2018.

\bibitem[DPG22]{DPMtwopoint}
G.~De~Philippis and M.~Goldman.
\newblock A two-point function approach to connectedness of drops in convex
  potentials.
\newblock {\em Comm. Anal. Geom.}, 30(4):815--841, 2022.

\bibitem[DRG19]{DG}
A.~De~Rosa and S.~Gioffr{\`e}.
\newblock Quantitative stability for anisotropic nearly umbilical
  hypersurfaces.
\newblock {\em J. Geom. Anal.}, 29(3):2318--2346, 2019.

\bibitem[DRG21]{DG2}
A.~De~Rosa and S.~Gioffr{\`e}.
\newblock Absence of bubbling phenomena for non-convex anisotropic nearly
  umbilical and quasi-{E}instein hypersurfaces.
\newblock {\em J. Reine Angew. Math.}, 2021(780):1--40, 2021.

\bibitem[DRKS20]{DRKS}
A.~De~Rosa, S.~Kolasi\'{n}ski, and M.~Santilli.
\newblock Uniqueness of critical points of the anisotropic isoperimetric
  problem for finite perimeter sets.
\newblock {\em Arch. Ration. Mech. Anal.}, 238(3):1157--1198, 2020.

\bibitem[DS02]{DS}
F.~Duzaar and K.~Steffen.
\newblock Optimal interior and boundary regularity for almost minimizers to
  elliptic variational integrals.
\newblock {\em J. Reine Angew. Math.}, 546:73--138, 2002.

\bibitem[DW22]{DelgadinoWeser}
M.~G. Delgadino and D.~Weser.
\newblock A {H}eintze--{K}archer inequality with free boundaries and
  applications to capillarity theory.
\newblock {\em Preprint at ar{X}iv:2210.16376}, 2022.

\bibitem[Fal10]{Fall}
M.~M. Fall.
\newblock Area-minimizing regions with small volume in {R}iemannian manifolds
  with boundary.
\newblock {\em Pacific J. Math.}, 244(2):235--260, 2010.

\bibitem[Fin82]{FinnDrops}
R.~Finn.
\newblock Global size and shape estimates for symmetric sessile drops.
\newblock {\em J. Reine Angew. Math.}, 335:9--36, 1982.

\bibitem[Fin86]{FinnBook}
R.~Finn.
\newblock {\em Equilibrium capillary surfaces}, volume 284 of {\em Grundlehren
  der mathematischen Wissenschaften [Fundamental Principles of Mathematical
  Sciences]}.
\newblock Springer-Verlag, New York, 1986.

\bibitem[FJ14]{FuJu}
N.~Fusco and V.~Julin.
\newblock A strong form of the quantitative isoperimetric inequality.
\newblock {\em Calc. Var. Partial Differential Equations}, 50(3-4):925--937,
  2014.

\bibitem[FM91]{fonsecam1991}
I.~Fonseca and S.~Müller.
\newblock A uniqueness proof for the {W}ulff theorem.
\newblock {\em Proceedings of the Royal Society of Edinburgh Section A:
  Mathematics}, 119(1-2):125–136, 1991.

\bibitem[FM11]{Figalli-Maggi-drops}
A.~Figalli and F.~Maggi.
\newblock On the shape of liquid drops and crystals in the small mass regime.
\newblock {\em Arch. Ration. Mech. Anal.}, 201(1):143--207, 2011.

\bibitem[FM13]{Figalli-Maggi-log}
A.~Figalli and F.~Maggi.
\newblock On the isoperimetric problem for radial log-convex densities.
\newblock {\em Calc. Var. Partial Differential Equations}, 48(3-4):447--489,
  2013.

\bibitem[FMP08]{FuMaPr}
N.~Fusco, F.~Maggi, and A.~Pratelli.
\newblock The sharp quantitative isoperimetric inequality.
\newblock {\em Ann. of Math. (2)}, 168(3):941--980, 2008.

\bibitem[FMP10]{FiMaPr}
A.~Figalli, F.~Maggi, and A.~Pratelli.
\newblock A mass transportation approach to quantitative isoperimetric
  inequalities.
\newblock {\em Invent. Math.}, 182(1):167--211, 2010.

\bibitem[GMT83]{GonMasTam83}
E.~Gonzalez, U.~Massari, and I.~Tamanini.
\newblock On the regularity of boundaries of sets minimizing perimeter with a
  volume constraint.
\newblock {\em Indiana Univ. Math. J.}, 32(1):25--37, 1983.

\bibitem[GNR22]{Gol22}
M.~Goldman, M.~Novaga, and B.~Ruffini.
\newblock Rigidity of the ball for an isoperimetric problem with strong
  capacitary repulsion.
\newblock {\em To appear in J. Eur. Math. Soc., preprint at arXiv:2201.04376},
  2022.

\bibitem[KM13]{KnMu1}
H.~Kn\"{u}pfer and C.~B. Muratov.
\newblock On an isoperimetric problem with a competing nonlocal term {I}: {T}he
  planar case.
\newblock {\em Comm. Pure Appl. Math.}, 66(7):1129--1162, 2013.

\bibitem[KM14]{KnMu2}
H.~Kn\"{u}pfer and C.~B. Muratov.
\newblock On an isoperimetric problem with a competing nonlocal term {II}:
  {T}he general case.
\newblock {\em Comm. Pure Appl. Math.}, 67(12):1974--1994, 2014.

\bibitem[KM17]{KM}
B.~Krummel and F.~Maggi.
\newblock Isoperimetry with upper mean curvature bounds and sharp stability
  estimates.
\newblock {\em Calc. Var. Partial Differential Equations}, 56(2):Paper No. 53,
  43, 2017.

\bibitem[Mag12]{Mag}
F.~Maggi.
\newblock {\em Sets of finite perimeter and geometric variational problems},
  volume 135 of {\em Cambridge Studies in Advanced Mathematics}.
\newblock Cambridge University Press, Cambridge, 2012.
\newblock An introduction to geometric measure theory.

\bibitem[McC98]{McCannPlanarCrystals}
R.~J. McCann.
\newblock Equilibrium shapes for planar crystals in an external field.
\newblock {\em Comm. Math. Phys.}, 195(3):699--723, 1998.

\bibitem[MJ00]{MoJo}
F.~Morgan and D.~L. Johnson.
\newblock Some sharp isoperimetric theorems for {R}iemannian manifolds.
\newblock {\em Indiana Univ. Math. J.}, 49(3):1017--1041, 2000.

\bibitem[MM16]{MaggiMihaila}
F.~Maggi and C.~Mihaila.
\newblock On the shape of capillarity droplets in a container.
\newblock {\em Calc. Var. Partial Differential Equations}, 55(5):Art. 122, 42,
  2016.

\bibitem[Mor03]{MorganIsop}
F.~Morgan.
\newblock Regularity of isoperimetric hypersurfaces in {R}iemannian manifolds.
\newblock {\em Trans. Amer. Math. Soc.}, 355(12):5041--5052, 2003.

\bibitem[MPPP07]{MPPP}
M.~Miranda, Jr., D.~Pallara, F.~Paronetto, and M.~Preunkert.
\newblock Heat semigroup and functions of bounded variation on {R}iemannian
  manifolds.
\newblock {\em J. Reine Angew. Math.}, 613:99--119, 2007.

\bibitem[MS86]{milman}
Vitali M. and Gideon S.
\newblock Asymptotic theory of finite dimensional normed spaces.
\newblock 1986.

\bibitem[Nar09]{Nardulli}
S.~Nardulli.
\newblock The isoperimetric profile of a smooth {R}iemannian manifold for small
  volumes.
\newblock {\em Ann. Global Anal. Geom.}, 36(2):111--131, 2009.

\bibitem[Neu16]{N16}
R.~Neumayer.
\newblock A strong form of the quantitative {W}ulff inequality.
\newblock {\em SIAM J. Math. Anal.}, 48(3):1727--1772, 2016.

\bibitem[NPST22]{novaga2022isoperimetric}
M.~Novaga, E.~Paolini, E.~Stepanov, and V.~M. Tortorelli.
\newblock Isoperimetric planar clusters with infinitely many regions.
\newblock {\em Preprint at arXiv:2210.05286}, 2022.

\bibitem[Ros05]{RosSurvey}
A.~Ros.
\newblock The isoperimetric problem.
\newblock In {\em Global theory of minimal surfaces}, volume~2 of {\em Clay
  Math. Proc.}, pages 175--209. Amer. Math. Soc., Providence, RI, 2005.

\bibitem[SS82]{SchoenSimon}
R.~Schoen and L.~Simon.
\newblock A new proof of the regularity theorem for rectifiable currents which
  minimize parametric elliptic functionals.
\newblock {\em Indiana Univ. Math. J.}, 31(3):415--434, 1982.

\bibitem[SSA77]{SSA}
R.~Schoen, L.~Simon, and F.~J. Almgren, Jr.
\newblock Regularity and singularity estimates on hypersurfaces minimizing
  parametric elliptic variational integrals. {I}, {II}.
\newblock {\em Acta Math.}, 139(3-4):217--265, 1977.

\bibitem[Tam84]{TamaniniDrops}
I.~Tamanini.
\newblock On the sphericity of liquid droplets.
\newblock Number 118, pages 235--241. 1984.
\newblock Variational methods for equilibrium problems of fluids (Trento,
  1983).

\bibitem[Tay78]{bams}
J.~E. Taylor.
\newblock {Crystalline variational problems}.
\newblock {\em Bull. Amer. Math. Soc.}, 84(4):568 -- 588, 1978.

\bibitem[Tom93]{Tomter}
P.~Tomter.
\newblock Constant mean curvature surfaces in the {H}eisenberg group.
\newblock In {\em Differential geometry: partial differential equations on
  manifolds ({L}os {A}ngeles, {CA}, 1990)}, volume~54 of {\em Proc. Sympos.
  Pure Math.}, pages 485--495. Amer. Math. Soc., Providence, RI, 1993.

\bibitem[Vol10]{UsefulThesis}
A.~Volkmann.
\newblock Regularity of isoperimetric hypersurfaces with obstacles in
  {R}iemannian manifolds.
\newblock {\em
  https://alexandervolkmann.files.wordpress.com/2014/10/diplom.pdf}, 2010.

\end{thebibliography}
\bibliographystyle{alpha}
\end{document}